\documentclass[12pt]{amsart}
\usepackage{amsmath,amsthm}
\usepackage{amssymb}
\usepackage{hyperref}
\newcommand{\R}{\mathbb R}
\newcommand{\C}{\mathbb C}
\newcommand{\N}{\mathbb N}
\newcommand{\g}{\mathfrak{g}}
\newtheorem{thm}{Theorem}[section]
\newtheorem{cor}[thm]{Corollary}
\newtheorem{lem}[thm]{Lemma}
\newtheorem{prop}[thm]{Proposition}
\theoremstyle{definition}
\newtheorem{defn}[thm]{Definition}

\newtheorem{rem}[thm]{Remark}
\numberwithin{equation}{section}

\begin{document}

\title[Boundary behaviour of positive solutions]{Boundary behavior of  positive solutions of the heat equation on a stratified Lie group}
\author[J. Sarkar]{Jayanta Sarkar}
\address{Stat Math Unit, Indian Statistical Institute, 203 B. T. Road, Kolkata 700108, India}
\email{jayantasarkarmath@gmail.com}
\maketitle
\begin{abstract}
In this article, we are concerned with a certain type of boundary behavior of positive solutions of the heat equation on a stratified Lie group at a given boundary point. We prove that a necessary and sufficient condition for the existence of the parabolic limit of a positive solution $u$ at a point on the boundary is the existence of the strong derivative of the boundary measure of $u$ at that point. Moreover, the parabolic limit and the strong derivative are equal. We also construct an example of a positive measure on the Heisenberg group to show that the set of all points where strong derivative exists is strictly larger than the set of Lebesgue points of the measure.
\end{abstract}
\section{Introduction}
To motivate our study in this paper, we first consider the heat equation
\begin{equation}\label{euclideanheat}
\Delta u(x,t)=\frac{\partial}{\partial t}u(x,t),
\end{equation}
on the Euclidean upper half space $\R^{n+1}_+=\{(x,t)\mid x\in\R^n,t>0\}$, where $\Delta=\sum_{i=1}^{n}\frac{\partial^2}{\partial x_i^2}$ is the Laplace operator on $\R^n$.  The fundamental solution of the heat equation is known as the Gauss-Weierstrass kernel or the heat kernel of $\R^{n+1}_+$ and is given by
\begin{equation*}
W(x,t)=(4\pi t)^{-\frac{n}{2}}e^{-\frac{\|x\|^2}{4t}},\:(x,t)\in\R^{n+1}_+.
\end{equation*}
In this article, by a measure $\mu$ we will always mean a complex Borel measure or a signed Borel measure such that the total variation $|\mu|$ is locally finite, that is, $|\mu|(K)$ is finite for all compact sets $K$. If $\mu(E)$ is nonnegative for all Borel measurable sets $E$ then $\mu$ will be called a positive measure. Also, by a positive solution of some partial differential equation, we shall always mean a nonnegative solution. The Gauss-Weierstrass integral of a measure $\mu$ on $\R^n$ is given by the convolution
\begin{equation*}
W\mu(x,t)=\int_{\R^n}W(x-y,t)\:d\mu(y),\:\:\:\: x\in\R^n,\:\: t\in (0,\infty ),
\end{equation*}
whenever the above integral exists. It is known that if $W\mu(x_0,t_0)$ is finite at some point $(x_0,t_0)\in\R^{n+1}_+$, then $W\mu$ is well defined and is a solution of the heat equation in $\{(x,t): x\in\R^n, t\in (0,t_0)\}$ \cite[Theorem 4.4]{W1}. On the other hand, it is also known \cite[P.93-99]{W1} that if $u$ is a positive solution of the heat equation (\ref{euclideanheat}) in $\R^{n+1}_+$, then there exists a unique positive measure $\mu$ (known as the boundary measure of $u$) on $\R^n$ such that
\begin{equation*}
u(x,t)=\int_{\R^n}W(x-\xi,t)\:d\mu(\xi),\:\:\:\:x\in\R^n,\:\: t>0.
\end{equation*}  
Given a function $u:\R^{n+1}_+\to\C$, we say that $u$ has parabolic limit $L\in\C$, at $x_0\in\R^n$, if for each $\alpha>0$,
\begin{equation*}
\lim_{\substack{(x,t)\to(x_0,0)\\\|x-x_0\|^2<\alpha t}}u(x,t)=L.
\end{equation*}

In \cite{G}, Gehring initiated the study of Fatou-type theorems and their converse for solutions of the heat equation for $n=1$. This was motivated by an earlier work of Loomis \cite{L} regarding converse of Fatou theorem for Poisson integral of positive measures. Gehring proved that if $u$ is a positive solution of the heat equation in $\R^2_+$ with boundary measure $\mu$ then $u$ has parabolic limit $L$ at $x_0$ if and only if $F_{\mu}^{\prime}(x_0)=L$, where $F_{\mu}$ is the distribution function of $\mu$.

It is not seemingly clear how to interpret the derivative $F_{\mu}^{\prime}$ in higher dimensions. Nevertheless, one possible way to resolve this problem is to consider the strong derivative of measures. The notion of strong derivative was introduced by Ramey-Ullrich \cite{UR} which we recall: a measure $\mu$ on $\R^n$ is said to have strong derivative $L\in\C$, at $x_0\in\R^n$, if
\begin{equation*}
\lim_{r\to 0}\frac{\mu(x_0+rB)}{|rB|}=L,
\end{equation*}
holds for every open ball $B\subset\R^n$. Here, $rE=\{rx\mid x\in E\}$, $r\in(0,\infty)$, $E\subset\R^n$, and $|A|$ denotes the Lebesgue measure of a measurable set $A\subset\R^n$.

One can show from an example constructed by Shapiro \cite{Sh} that there are measures for which the set of all points where strong derivative exists is strictly larger than that of all Lebesgue points of the measure. The following is a higher dimensional analogue of Gehring's result.
\begin{thm}\label{thmparaeucl}
Suppose that $u$ is a positive solution of the heat equation (\ref{euclideanheat}) and that $x_0\in\R^n$, $L\in[0,\infty)$. If $\mu$ is the boundary measure of $u$ then the following statements are equivalent.
\begin{enumerate}
\item[i)]The strong derivative of $\mu$ at $x_0$ is equal to $L$.
\item[ii)]$u$ has parabolic limit $L$ at $x_0$.
\end{enumerate}
\end{thm}
The proof of Theorem \ref{thmparaeucl} can be found in \cite[Theorem 3.2]{Sar} (see also \cite{BC}) and it is inspired by the work of Ramey and Ullrich \cite[Theorem 2.3]{UR} on the nontangential convergence of positive harmonic functions in $\R^{n+1}_+$. It is worth pointing out that a recent result of Bar \cite[Theorem 4]{B} on generalization of Montel's theorem plays an important role in the proof of Theorem \ref{thmparaeucl} given in \cite{Sar}.

In this article, our aim is to prove an analogue of Theorem \ref{thmparaeucl} for positive solutions of more genereal partial differential equations. We will consider  positive solutions $u$ of the heat equation corresponding to a sub-Laplacian on a stratified Lie group and prove the equivalence between the parabolic convergence of $u$ to the boundary and the strong derivative of the boundary measure of $u$ (Theorem \ref{mainc}). We refer the reader to Definition \ref{impdefnc}, for the relevant definitions. One of the main difficulty in this setting is that we do not have any explicit expression of the fundamental solution or the heat kernel. However, we do have Gaussian estimates of the heat kernels (see Theorem \ref{fundamental}) and using this estimate we have been able to prove our results.

This paper is organised as follows: In Section 2, we will collect some basic information about stratified Lie groups and the heat equation on these groups. In Section 3, We construct an example of a positive measure on the Heisenberg group to show that the set of all points where strong derivative exists strictly conatains that of Lebesgue points of the measure. Proofs of a result about heat maximal functions, and other relevant Lemmas are given in Section 4. The statement and proof of the main theorem (Theorem \ref{mainc}) is given in the last section.

\section{Preliminaries about stratified Lie groups}
Stratified Lie groups (also known as Carnot groups) have been introduced by Folland \cite[P.162]{F}. The simplest nontrivial example of stratified group is the Heisenberg group $H^n$. We will discuss about them in the next section. A stratified Lie group $(G,\circ)$ is a connected, simply connected nilpotent Lie group whose Lie algebra $\mathfrak{g}$ admits a vector space decomposition
\begin{equation*}
\mathfrak{g}=V_1\oplus V_2\oplus\cdots\oplus V_l,
\end{equation*}
such that
\begin{equation*}
[V_1,V_j]=V_{j+1},\:\:1\leq j<l,\:\:\:\:\:\:\:\:[V_1,V_l]=0.
\end{equation*}
Here,
\begin{equation*}
[V_1,V_j]=\text{span}~\{[X,Y]\mid X\in V_1, Y\in V_j \}.
\end{equation*}
Therefore, $V_1$ generates $\mathfrak{g}$ as a Lie algebra. We say that $G$ is of step $l$ and has $\text{dim}\:V_1$ many generators. The Lie algebra $\mathfrak{g}$ is eqquiped with a cannonical family of dilations $\{\delta_r\}_{r>0}$ which are Lie algebra automophisms defined by \cite[P.5]{FS}
\begin{equation*}
\delta_r\left(\sum_{j=1}^{l}X_j\right)=\sum_{j=1}^{l}r^jX_j,\:\:X_j\in V_j.
\end{equation*}
Since $\mathfrak{g}$ is nilpotent, the exponential map $\exp:\mathfrak{g}\to G$ is a diffeomorphism, and the dilations $\delta_r$ lift via the exponential map to give a one-parameter group of automorphisms of $G$ which we still denote by $\delta_r$. We fix once and for all a bi-invariant measure $m$ on $G$ which is the push forward of the Lebesgue measure on $\mathfrak{g}$ via the exponential map. The bi-invariant measure $m$ on $G$ is, in fact, the Lebesgue measure of the underlying Euclidean space. We shall denote by
\begin{equation*}
Q=\sum_{j=1}^lj(\text{dim}\:V_j),
\end{equation*}
the homogeneous dimension of $G$ and by $0$ the identity element of $G$. The importance of homogeneous dimension stems from the following relation 
\begin{equation*}
m\left(\delta_r(E)\right)=r^Qm(E),
\end{equation*}
which holds for all measurable sets $E\subset G$ and $r>0$. A homogeneous norm on $G$ is a continuous function $d:G\to[0,\infty)$ satisfying the following
\begin{enumerate}
	\item[i)]$d$ is smooth on $G\setminus\{0\}$;
	\item[ii)] $d(\delta_r(x))=rd(x)$, for all $r>0,\:x\in G$;
	\item [iii)]$d(x^{-1})=d(x)$, for all $x\in G$;
	\item[iv)]$d(x)=0$ if and only if $x=0$.
\end{enumerate}
It is known \cite[P.8-10]{FS} that homogeneous norms always exist on stratified Lie groups and for any homogeneous norm $d$ on $G$ there exists a positive constant $C_d$ such that 
\begin{equation}\label{quasimorm}
d(x\circ y)\leq C_d\{d(x)+d(y)\},\;\:\:\:x\in G, y\in G.
\end{equation}
Moreover, any two homogeneous norms on $G$ are equivalent (see \cite[P.230]{BLU}): if $d_1$ and $d_2$ are two homogeneous norms on $G$ then there exists a positive constant $B$ such that 
\begin{equation*}
B^{-1}d_1(x)\leq d_2(x)\leq Bd_2(x),\:\:\:\text{for all}\:\:x\in G.
\end{equation*}
The homogeneous norm $d$ defines a left invariant quasi-metric on $G$ (again denoted by $d$) as follows:
\begin{equation*}
d(x,y)=d(x^{-1}\circ y),\:\:\:\: x\in G,y\in G.
\end{equation*}
One can then write from (\ref{quasimorm}) that
\begin{equation*}
d(x,y)\leq C_d\left(d(x,z)+d(z,y)\right),\:\:\:\text{for all}\:\:x,\:y,\:z\in G.
\end{equation*}
\begin{rem}\label{topology} (\cite[Proposition 3.5]{Le})
Every homogeneous norm on $G$ induces the Euclidean topology on $G$.
\end{rem}
\begin{rem}(\cite[Proposition 5.15.1]{BLU})
Let $d$ be a homogeneous norm on $G$. Then for every compact set $K\subset G$, there exists a positive constant $c_K$ such that
\begin{equation}\label{bilipschitz}
(c_K)^{-1}\|x-y\|\leq d(y^{-1}\circ x)\leq c_K\|x-y\|^{\frac{1}{l}},\:\:\:\:\text{for all}\:\:x,\:y\in K,
\end{equation}
where $l$ is the step of $G$ and $\|\cdot\|$ is the norm of the underlying Euclidean space.
\end{rem}
For $x\in G$ and $s>0$,  the $d$-ball centered at $x$ with radius $s$ is defined as
\begin{equation*}
B_d(x,s)=\{y\in G:d(x,y)<s\}.
\end{equation*}
It follows that $B_d(x,s)$ is the left translate by $x$ of the ball $B_d(0,s)$ which in turn, is the image under $\delta_s$ of the ball $B_d(0,1)$. 
\begin{rem}\label{compact}(\cite[Lemma 1.4]{FS})
If $B=B_d(x,s)$ is a $d$-ball then $\overline{B}=\{y\in G:d(x,y)\leq s\}$ is compact with respect to the Euclidean topology of $G$.
\end{rem} 

We identify $\mathfrak{g}$ as the Lie algebra of all left invariant vector fields on $G$ and fix once and for all a basis $\{X_1,X_2,\cdots,X_{N_1}\}$ for $V_1$, with $N_1=\text{dim}\:V_1$, which generates $\mathfrak{g}$ as a Lie algebra. The second order differential operator
\begin{equation*}
\mathcal{L}=\sum_{j=1}^{N_1}X_j^2,
\end{equation*}
is called a sub-Laplacian for $G$.
\begin{rem} (\cite[Theorem 2.2]{BLU1})
There exists a homogeneous norm $d_{\mathcal{L}}$ on $G$ such that $d_{\mathcal{L}}(\cdot )^{2-Q}$ is the fundamental solution of $\mathcal{L}$.
\end{rem}
A differential operator $D$ on $G$ is said to be homogeneous of degree $\lambda$, where $\lambda\in\C$, if
\begin{equation*}
D(f\circ\delta_r)=r^{\lambda}(Df)\circ\delta_r,
\end{equation*}
for all $f\in C_c^{\infty}(G)$, $r>0$. It is evident that $X\in \mathfrak{g}$ is homogeneous of degree $j$ if and only if $X\in V_j$, $1\leq j\leq k$ (see \cite[P.172]{F}). Hence, $\mathcal{L}$ is a left invariant second order differential operator on $G$ which is homogeneous of degree two.
The heat operator $\mathcal{H}$ associated to the sub-Laplacian $\mathcal{L}$ is the differential operator
\begin{equation*}
\mathcal{H}=\mathcal{L}-\frac{\partial}{\partial t}
\end{equation*}
on $G\times(0,\infty)$. Since $X_1,X_2,\cdots,X_{N_1}$ generates the whole $\g$ as an algebra, by a celebrated theorem of Hörmander \cite[Theorem 1.1]{H}, $\mathcal{L}$ and $\mathcal{H}$ are hypoelliptic on  $G$ and $G\times (0,\infty)$ respectively. Hypoellipticity of $\mathcal{H}$ plays an important role in the results we have proved.

As stated before, in this paper, we are interested in boundary behavior of positive solutions of the heat equation on stratified groups:
\begin{equation}\label{heateq}
\mathcal{H}u(x,t)=0,\:\:\:\:\:\:(x,t)\in G\times (0,\infty).
\end{equation}
We list down some properties of the fundamental solution of the heat equation (\ref{heateq}).
\begin{thm}\label{fundamental}
	The fundamental solution of $\mathcal{H}$ is given by
	\begin{equation*}
	\Gamma(x,t;\xi):=\Gamma(\xi^{-1}\circ x,t),\:\:\:\: x\in G,\:\xi\in G,\:t\in(0,\infty),
	\end{equation*}where $\Gamma$ is a smooth, strictly positive function on $G\times(0,\infty)$ satisfying the following properties:
	\begin{enumerate}
		\item[(i)] $\Gamma(x,t+\tau)=\int_{G}\Gamma(\xi^{-1}\circ x,t)\Gamma(\xi,\tau)\:dm(\xi)$,\: \:\:$x\in G,\:\:t\in(0,\infty),\:\:\tau\in(0,\infty)$.
		\item[(ii)] $\Gamma(x,t)=\Gamma(x^{-1},t)$,\:\:\:$(x,t)\in G\times(0,\infty)$.
		\item[(iii)] $\Gamma(\delta_r(x),r^2t)=r^{-Q}\Gamma(x,t)$,\:\:\:$(x,t)\in G\times(0,\infty)$,\:\:$r>0$.
		\item[(iv)] $\int_{G}\Gamma(x,t)\:dm(x)=1$,\:\:\:$t>0$.
		\item [(v)] There exists a positive constant $c_0$, depending only on $\mathcal{L}$, such that the following Gaussian estimates hold.
		\begin{equation}\label{gaussian}
		c_0^{-1}t^{-\frac{Q}{2}}\exp\left(-\frac{c_0d_{\mathcal{L}}(x)^2}{t}\right)\leq\Gamma(x,t)\leq c_0t^{-\frac{Q}{2}}\exp\left(-\frac{d_{\mathcal{L}}(x)^2}{c_0t}\right),
		\end{equation} for every $x\in G$ and $t>0$.
		\item[(vi)] Given any nonnegative integers $p$, $q$, there exists positive constants $c_1$, $c_{p,q}$ such that for every $i_1,\cdots,i_{p}\in\{1,\cdots,N_1\}$ we have 
		\begin{equation}\label{derivativeest}
		|X_{i_1}\cdots X_{i_p}(\partial_t)^q\Gamma(x,t)|\leq c_{p,q}t^{-\frac{Q+p+2q}{2}}\exp\left(-\frac{d_{\mathcal{L}}(x)^2}{c_1t}\right),
		\end{equation}
		for every $x\in G$ and $t>0$.
	\end{enumerate}
\end{thm}
The proof of (i)-(iv) can be found in \cite[Proposition 1.68, Corollary 8.2]{FS} and the proofs of (v), (vi) are available in \cite[Theorem 5.1, Theorem 5.2, Theorem 5.3]{BLU1}. Property (v) plays an important role in our study and we will frequently
use it throughout this paper. We refer the reader to \cite[P.185-198]{FR} for other interesting properties of the heat kernel on a stratified Lie group.

For a measure $\mu$ on $G$, we define 
\begin{equation}\label{gammamu}
\Gamma\mu(x,t)=\int_{G}\Gamma(\xi^{-1}\circ x,t)\:d\mu(\xi),\:\:\:\:\:x\in G,\:\:t>0,
\end{equation} 
whenever the integral above exists. We define 
\begin{equation*}
\gamma(x):=\Gamma(x,1),\:\:\:\:\:x\in G.
\end{equation*}
Then by Therorem \ref{fundamental}, (iii) we have 
\begin{equation*}
\Gamma(x,t)=t^{-\frac{Q}{2}}\gamma\left(\delta_{\frac{1}{\sqrt{t}}}(x)\right)
\end{equation*}
For a function $\psi$ defined on $G$, we set for $t>0$,
\begin{equation*}
\psi_t(x)=t^{-Q}\psi\left(\delta_{\frac{1}{t}}(x)\right),\:\:\:\:\:x\in G.
\end{equation*}
Hence, we can rewrite (\ref{gammamu}) as follows.
\begin{equation}\label{gammaconv}
\Gamma\mu(x,t)=\mu\ast\gamma_{\sqrt{t}}(x),\:\:\:\:\:x\in G,\:\:t\in(0,\infty).
\end{equation}
where $\ast$ is the convolution on the group $G$. From now onwards, unless mentioned explicitly, we will always write $B(x,s)$ instead of $B_{d_{\mathcal{L}}}(x,s)$ to denote a ball centered at $x$ and radius $s>0$, with respect to the homogeneous norm $d_{\mathcal{L}}$. We recall that there exists a constant $C_\mathcal{L}\geq 1$, such that
\begin{equation*}
d_{\mathcal{L}}(y\circ z)\leq C_{\mathcal{L}}\left(d_{\mathcal{L}}(y)+d_{\mathcal{L}}(z)\right),\:\:y,\:z\in G.
\end{equation*}
Using this we get the following simple inequality.
\begin{equation}\label{revtri}
d_{\mathcal{L}}(y,z)\geq\frac{1}{C_{\mathcal{L}}}d_{\mathcal{L}}(u,z)-d_{\mathcal{L}}(u,y),\:\:u,\:y,\:z\in G.
\end{equation} 
We next prove a simple lemma regarding convolution on $G$. To do this, we take a function $\phi:G\rightarrow (0,\infty)$ such that
\begin{equation}\label{lradial}
\phi(x_1)=\phi(x_2),\:\:\text{whenever}\:\:d_{\mathcal{L}}(x_1)=d_{\mathcal{L}}(x_2);
\end{equation}
\begin{equation}\label{ldec}
\phi(x_1)\leq\phi(x_2),\:\:\text{whenever}\:\:d_{\mathcal{L}}(x_1)\geq d_{\mathcal{L}}(x_2).
\end{equation}
Following \cite[P.247]{BLU}, any function satisfying (\ref{lradial}) (resp.(\ref{ldec})) will be called $\mathcal{L}$-radial (resp. $\mathcal{L}$-radially decreasing) function. If $\phi$ is $\mathcal{L}$-radial function on $G$, for the sake of simplicity, we shall often write $\phi(r)=\phi(x)$ whenever $r=d_{\mathcal{L}}(x)$.
\begin{prop}\label{finiteconv}
Suppose that $\mu$ is a measure on $G$ and that $\phi$ is as above. Then finiteness of $|\mu|\ast \phi_{t_0}(x_0)$ for some $(x,t)\in G\times(0,\infty)$ implies the finiteness of $|\mu|\ast\phi_t(x)$ for all $(x,t)\in G\times (0,t_0/C_\mathcal{L})$.
\end{prop}  
\begin{proof}
We take $(x,t)\in G\times(0,t_0/C_\mathcal{L})$ and set $\alpha=\frac{t_0}{t_0-tC_\mathcal{L}}$. We write
\begin{eqnarray}\label{intsplit}
|\mu|\ast\phi_t(x)&=&t^{-Q}\int_{\{\xi\in G:\:d_{\mathcal{L}}(\xi,x_0)<\alpha C_\mathcal{L} d_{\mathcal{L}}(x,x_0)\}}\phi\left(\delta_{\frac{1}{t}}(\xi^{-1}\circ x)\right)\:d|\mu|(\xi)\nonumber\\ &&\:\:\:+t^{-Q}\int_{\{\xi\in G:\:d_{\mathcal{L}}(\xi,x_0)\geq\alpha C_\mathcal{L} d_{\mathcal{L}}(x,x_0)\}}\phi\left(\delta_{\frac{1}{t}}(\xi^{-1}\circ x)\right)\:d|\mu|(\xi)\nonumber\\
&\leq& t^{-Q}\phi(0)|\mu|\left(B(x_0,\alpha C_\mathcal{L} d_{\mathcal{L}}(x,x_0))\right)\\\nonumber&&\:\:\:+t^{-Q}\int_{\{\xi\in G:\:d_{\mathcal{L}}(\xi,x_0)\geq\alpha C_\mathcal{L} d_{\mathcal{L}}(x,x_0)\}}\phi\left(\delta_{\frac{1}{t}}(\xi^{-1}\circ x)\right)\:d|\mu|(\xi).
\end{eqnarray}
Using the reverse triangle inequality (\ref{revtri}), we obtain
\begin{equation*}
d_{\mathcal{L}}(\xi,x)\geq\frac{1}{C_{\mathcal{L}}}d_{\mathcal{L}}(\xi,x_0)-d_{\mathcal{L}}(x,x_0)\geq\left(\frac{1}{C_\mathcal{L}}-\frac{1}{\alpha C_\mathcal{L}}\right)d_{\mathcal{L}}(\xi,x_0),
\end{equation*}
whenever $d_{\mathcal{L}}(\xi,x_0)\geq\alpha C_\mathcal{L} d_{\mathcal{L}}(x,x_0)$. Therefore,
\begin{equation*}
d_{\mathcal{L}}\left(\delta_{\frac{1}{t}}(\xi^{-1}\circ x)\right)\geq\frac{1}{t}\left(\frac{1}{C_\mathcal{L}}-\frac{1}{\alpha C_\mathcal{L}}\right)d_{\mathcal{L}}\left(\xi^{-1}\circ x_0\right)=\frac{1}{t_0}d_{\mathcal{L}}\left(\xi^{-1}\circ x_0\right)=d_{\mathcal{L}}\left(\delta_{\frac{1}{t_0}}(\xi^{-1}\circ x_0)\right),
\end{equation*}
whenever $d_{\mathcal{L}}(\xi,x_0)\geq\alpha C_\mathcal{L} d_{\mathcal{L}}(x,x_0)$. Using this observation, and the fact that $\phi$ is ${\mathcal{L}}$-radially decreasing, we get from (\ref{intsplit})
\begin{equation*}
|\mu|\ast\phi_t(x)\leq t^{-Q}\phi(0)|\mu|\left(B(x_0,\alpha C_\mathcal{L} d_{\mathcal{L}}(x,x_0))\right)+ t^{-Q}\int_{G}\phi\left(\delta_{\frac{1}{t_0}}(\xi^{-1}\circ x_0)\right)\:d|\mu|(\xi).
\end{equation*}
By our hypothesis, integral on the right-hand side is finite and hence $|\mu|\ast\phi_t(x)<\infty$. 
\end{proof}
Using this Proposition and the Gaussian estimates (\ref{gaussian}), (\ref{derivativeest}) we can prove the following.
\begin{cor}\label{finiteonstrip}
Suppose $\mu$ is a measure on $G$. If $\Gamma\mu(x_0,t_0)$ exists for some $(x_0,t_0)\in G\times(0,\infty)$ then  $\Gamma\mu$ is well defined on the whole strip $G\times(0,\delta)$, where $\delta=\frac{t_0}{2c_0^2C_{\mathcal{L}}}$. Moreover, $\Gamma\mu$ is a solution of $\mathcal{H}u=0$ in $G\times(0,\delta)$.
\end{cor}
\begin{proof}
As $\Gamma\mu(x_0,t_0)$ exists, using (\ref{gaussian}) we get
\begin{equation}
\int_{G}\exp\left(-\frac{c_0d_{\mathcal{L}}(\xi^{-1}\circ x_0)^2}{t_0}\right)\:d|\mu|(\xi)<\infty.
\end{equation}
Consequently, for all $t\in (0,t_0/c_0^2)$
\begin{equation}\label{int1lemma1}
\int_{G}\exp\left(-\frac{d_{\mathcal{L}}(\xi^{-1}\circ x_0)^2}{c_0t}\right)\:d|\mu|(\xi)<\infty.
\end{equation}
Setting $\phi(x)=\exp\left(\frac{-d_{\mathcal{L}}(x)^2}{c_0}\right)$, we note that $\phi$ satisfies all the requirements of Proposition \ref{finiteconv}. Moreover, by (\ref{int1lemma1})
\begin{equation*}
|\mu|\ast\phi_{\sqrt{t_1}}(x_0)<\infty,
\end{equation*}
where $t_1=\frac{t_0}{2c_0^2}$. Applying Proposition \ref{finiteconv}, we conclude that $|\mu|\ast\phi_{\sqrt{t}}(x)<\infty$, for all $x\in G$, $t\in (0,t_1/C_{\mathcal{L}})$. Consequently, it follows from the Gaussian estimate (\ref{gaussian}) that $\Gamma\mu(x,t)$ is well defined for all $(x,t)\in G\times (0,\frac{t_0}{2c_0^2C_{\mathcal{L}}})$. To prove the second part, we differentiate $\Gamma\mu$ in $G\times(0,\delta)$ along the vector fields $X_1,\cdots,X_{N_1},\frac{\partial}{\partial t}$ and then use the fact that $\Gamma$ is a fundamental solution of $\mathcal{H}$. Differentiation under integral sign is justified because of the estimate (\ref{derivativeest}).
\end{proof}
\begin{rem}
For an alternative proof of the second part of this Corollary \ref{finiteonstrip}, which uses Harnack inequality, we refer to \cite[Lemma 2.5]{BU}.
\end{rem}
It is clear from the Gaussian estimate (\ref{gaussian}) that for each $t>0$, $\Gamma(\cdot,t)\in L^p(G)$, for all $p\in [1,\infty]$. Thus, for $f\in L^p(G)$, $1\leq p\leq\infty$,
\begin{equation*}
\Gamma f(x,t):=\int_G\Gamma(\xi^{-1}\circ x,t)f(\xi)\:dm(\xi)
\end{equation*}
is well defined for all $(x,t)\in G\times(0,\infty)$. This follows from the formula for integration in ``polar coordinates" \cite[Proposition 1.15]{FS}: for all $g\in L^1(G)$,
\begin{equation}\label{polarcordinate}
\int_Gg(x)\:dm(x)=\int_{0}^{\infty}\int_{S}g(\delta_r(\omega))r^{Q-1}\:d\sigma(\omega)\:dr,
\end{equation}
where $S=\{\omega\in G:d_{\mathcal{L}}(\omega)=1\}$ and $\sigma$ is a unique positive Radon measure on $S$.
\begin{rem}\label{appcc} (\cite[Proposition 1.20]{FS})
As $\gamma$ is positive with $\int_{G}\gamma(x)\:dm(x)=1$ and $\Gamma f(.,t)=f\ast\gamma_{\sqrt{t}}$, it can be shown that if $f\in C_c(G)$ then $\Gamma f(.,t)$ converges to $f$ uniformly as $t$ goes to zero.
\end{rem}
However, a stronger result is true.
\begin{prop}\label{unifc}
If $f\in C_c(G)$ then
\begin{equation*}
\lim_{t\to 0}\frac{\Gamma f(.,t)}{\gamma}=\frac{f}{\gamma},
\end{equation*}
uniformly on $G$.
\end{prop}\label{ccfun}
\begin{proof}
We assume that $\text{supp}\:f\subset B(0,R)$ for some $R>0$ . By (\ref{gaussian}), $\gamma$ is bounded below by a positive number on $B(0,2RC_{\mathcal{L}})$. Therefore, Remark \ref{appcc} tells us that 
\begin{equation*}
\lim_{t\to 0}\frac{\Gamma f(x,t)}{\gamma(x)}=\frac{f(x)}{\gamma(x)},
\end{equation*}
uniformly for  $x\in B(0,2RC_{\mathcal{L}})$. Hence, it suffices to prove that
\begin{equation*}
\lim_{t\to 0}\frac{\Gamma f(x,t)}{\gamma(x)}=0,
\end{equation*}
uniformly for $x\in G\setminus B(0,2RC_{\mathcal{L}})$.
We observe that
\begin{eqnarray}\label{gammaf}
\frac{|\Gamma f(x,t)|}{\gamma(x)}&=&\frac{1}{\gamma(x)}\left|\int_{G}\Gamma(\xi^{-1}\circ x,t)f(\xi)\:dm(\xi)\right|\nonumber\\
&\leq&\frac{1}{\gamma(x)}\int_{B(0,R)}c_0t^{-\frac{Q}{2}}\exp\left(-\frac{d_{\mathcal{L}}(\xi^{-1}\circ x)^2}{c_0t}\right)|f(\xi)|\:dm(\xi),
\end{eqnarray}
where last inequality follows from the Gaussian estimate (\ref{gaussian}) and the fact that $\text{supp} f\subset B(0,R)$. Now, for $x\in G\setminus B(0,2RC_{\mathcal{L}})$ and $\xi\in B(0,R)$, using (\ref{revtri}) we get
\begin{equation}\label{atrineq}
d_{\mathcal{L}}(\xi^{-1}\circ x)\geq \frac{d_{\mathcal{L}}(x)}{C_{\mathcal{L}}}-d_{\mathcal{L}}(\xi)\geq \frac{d_{\mathcal{L}}(x)}{C_{\mathcal{L}}}-\frac{d_{\mathcal{L}}(x)}{2C_{\mathcal{L}}}=\frac{d_{\mathcal{L}}(x)}{2C_{\mathcal{L}}}.
\end{equation}
Using this in (\ref{gammaf}), we obtain
\begin{equation*}
\frac{|\Gamma f(x,t)|}{\gamma(x)}\leq \frac{c_0}{t^{\frac{Q}{2}}\gamma(x)}\int_{B(0,R)}\exp\left(-\frac{d_{\mathcal{L}}(x)^2}{4c_0C_{\mathcal{L}}^2t}\right)|f(\xi)|\:dm(\xi)\leq \frac{c_0\exp\left(-\frac{d_{\mathcal{L}}(x)^2}{4c_0C_{\mathcal{L}}^2t}\right)}{t^{\frac{Q}{2}}\gamma(x)}\|f\|_{L^1(G)}.
\end{equation*}
Hence, it is enough to show that
\begin{equation*}
\lim_{t\to 0}\frac{\exp\left(-\frac{d_{\mathcal{L}}(x)^2}{4c_0C_{\mathcal{L}}^2t}\right)}{t^{\frac{Q}{2}}\gamma(x)}=0,
\end{equation*}
uniformly for $x\in G\setminus B(0,2RC_{\mathcal{L}})$. But 
\begin{equation*}
\frac{\exp\left(-\frac{d_{\mathcal{L}}(x)^2}{4c_0C_{\mathcal{L}}^2t}\right)}{t^{\frac{Q}{2}}\gamma(x)}
\leq\frac{\exp\left(-\frac{d_{\mathcal{L}}(x)^2}{4c_0C_{\mathcal{L}}^2t}\right)}{t^{\frac{Q}{2}}c_0^{-1}\exp(-c_0d_{\mathcal{L}}(x)^2)}
=c_0t^{-\frac{Q}{2}}\exp\left(-\left(\frac{1}{4c_0C_{\mathcal{L}}^2t}-c_0\right){d_{\mathcal{L}}(x)^2}\right),
\end{equation*}
where the inequality follows from the Gaussian estimate (\ref{gaussian}). Taking
$0<t<\frac{1}{4c_0^2C_{\mathcal{L}}^2}$, we see that $\frac{1}{4c_0C_{\mathcal{L}}^2t}-c_0$ is positive. Hence, for such $t$ and for all $x\in G\setminus B(0,2RC_{\mathcal{L}})$ that is, $d_{\mathcal{L}}(x)\geq2RC_{\mathcal{L}}$, last inequality gives
\begin{equation*}
\frac{\exp\left(-\frac{d_{\mathcal{L}}(x)^2}{4c_0C_{\mathcal{L}}^2t}\right)}{t^{\frac{Q}{2}}\gamma(x)}\leq c_0t^{-\frac{Q}{2}}\exp\left(-\left(\frac{1}{4c_0C_{\mathcal{L}}^2t}-c_0\right){4C_{\mathcal{L}}^2R^2}\right)\leq At^{-\frac{Q}{2}}e^{-\frac{1}{Bt}},
\end{equation*}
for some positive constants $A$ and $B$. The expression on the right-hand side of the inequality above goes to zero as $t$ goes to zero. This completes the proof.
\end{proof}
Let $M$ denote the set of all measures $\mu$ on $G$ such that  $\Gamma\mu$ exists on $G\times(0,\infty)$. In view of Corollary \ref{finiteonstrip}, we have
\begin{equation*}
M=\{\mu\:\text{is a measure on}\:G\mid \Gamma\mu(0,t)\text{ exists for all $t\in (0,\infty)$}\}.
\end{equation*}
We note that if $|\mu|(G)<\infty$, then $\mu\in M$. In particular, every complex measure on $G$ belongs to $M$. We have the following observation regarding this class of measures.
\begin{prop}\label{fubinic}
If $\nu\in M$ and $f\in C_c(G)$ then for each fixed $t>0$,
\begin{equation*}
\int_{G}\Gamma f(x,t)\:d\nu(x)=\int_{G}\Gamma\nu(x,t)f(x)\:dm(x).
\end{equation*}
\end{prop}
\begin{proof}
The result will follow by interchanging integrals using Fubini's theorem. In order to apply Fubini's theorem we must prove that 
\begin{equation*}
\int_{G}\int_{\text{supp}\:f}\Gamma(\xi^{-1}\circ x,t)|f(\xi)|\:dm(\xi)\:d|\nu|(x)<\infty.
\end{equation*}
We asuume that $\text{supp}\:f\subset B(0,R)$, for some $R>0$. Then for each fixed $t>0$,
\begin{eqnarray*}
I&:=&\int_{G}\int_{B(0,R)}\Gamma(\xi^{-1}\circ x,t)|f(\xi)|\:dm(\xi)\:d|\nu|(x)\nonumber\\
&\leq&c_0t^{-\frac{Q}{2}}\int_{G}\int_{B(0,R)}\exp\left(-\frac{d_{\mathcal{L}}(\xi^{-1}\circ x)^2}{c_0t}\right)|f(\xi)|\:dm(\xi)\:d|\nu|(x)\:\:\:\:\:(\text{using}\:(\ref{gaussian}))\nonumber\\
&=&c_0t^{-\frac{Q}{2}}\int_{B(0,2C_{\mathcal{L}}R)}\int_{B(0,R)}\exp\left(-\frac{d_{\mathcal{L}}(\xi^{-1}\circ x)^2}{c_0t}\right)|f(\xi)|\:dm(\xi)\:d|\nu|(x)\nonumber\\
&&\:\:\:+c_0t^{-\frac{Q}{2}}\int_{G\setminus B(0,2C_{\mathcal{L}}R)}\int_{B(0,R)}\exp\left(-\frac{d_{\mathcal{L}}(\xi^{-1}\circ x)^2}{c_0t}\right)|f(\xi)|\:dm(\xi)\:d|\nu|(x)\nonumber\\
&\leq&c_0t^{-\frac{Q}{2}}|\nu|(B(0,2C_{\mathcal{L}}R))\|f\|_{L^1(G)}\nonumber\\
&&\:\:\:+c_0t^{-\frac{Q}{2}}\int_{G\setminus B(0,2C_{\mathcal{L}}R)}\int_{B(0,R)}\exp\left(-\frac{d_{\mathcal{L}}( x)^2}{4c_0C_{\mathcal{L}}^2t}\right)|f(\xi)|\:dm(\xi)\:d|\nu|(x),\nonumber\\
\end{eqnarray*}
where we have used (\ref{atrineq}) in the last integral. Applying Gaussian estimate (\ref{gaussian}) in the last integral, we get
\begin{eqnarray*}
I&\leq&c_0t^{-\frac{Q}{2}}|\nu|(B(0,2C_{\mathcal{L}}R))\|f\|_{L^1(G)}\\
&&\:\:\:+(4c_0^2C_{\mathcal{L}}^2)^{\frac{Q}{2}}c_0^2\int_{G\setminus B(0,2C_{\mathcal{L}}R)}\int_{B(0,R)}\Gamma(x,4c_0^2C_{\mathcal{L}}^2t)|f(\xi)|\:dm(\xi)\:d|\nu|(x)\\
&\leq&c_0t^{-\frac{Q}{2}}|\nu|(B(0,2C_{\mathcal{L}}R))\|f\|_{L^1(G)}+(4c_0^2C_{\mathcal{L}}^2)^{\frac{Q}{2}}c_0\|f\|_{L^1(G)}\Gamma\mu(0,4c_0^2C_{\mathcal{L}}^2t)
\end{eqnarray*} 
As $\nu\in M$, $I$ is finite. This proves the lemma.
\end{proof}
Before we move into our next section, we end this section with some definitions that will be used in the upcoming sections.
\begin{defn}\label{impdefnc}
\begin{enumerate}
\item[i)] A function $u$ defined on $G\times(0,t_0)$, for some $<t_0\leq\infty$ is said to have parabolic limit $L\in\C$, at $x_0\in G$ if for each $\alpha>0$
\begin{equation*}
\lim_{\substack{(x,t)\to(x_0,0)\\(x,t)\in \texttt{P}(x_0,\alpha)}}u(x,t)=L,
\end{equation*}
where $\texttt{P}(x_0,\alpha)=\{(x,t)\in G\times(0,\infty):d_{\mathcal{L}}(x_0,x)<\alpha\sqrt{t}\}$ is the parabolic domain with vertex at $x_0$ and aperture $\alpha$.
\item[ii)] Given a measure $\mu$ on $G$, we say that $\mu$ has strong derivative $L\in[0,\infty)$ at $x_0$ if
\begin{equation*}
\lim_{r\to 0}\frac{\mu(x_0\circ\delta_r(B))}{m(x_0\circ\delta_r(B))}=L
\end{equation*}
holds for every $d_{\mathcal{L}}$-ball $B\subset G$. The strong derivative of $\mu$ at $x_0$, if it exists, is denoted by $D\mu(x_0)$. Note that if $B=B(y,s)$ for some $y\in G$, $s>0$, then $\delta_r(B)=B(\delta_r(y),rs)$, for all $r>0$.
\item[iii)] A sequence of functions $\{u_j\}$ defined on $G\times(0,\infty)$ is said to converge normally to a function $u$ if $\{u_j\}$ converges to $u$ uniformly on compact subsets of $G\times(0,\infty)$.
\item[iv)] A sequence of functions $\{u_j\}$ defined on $G\times(0,\infty)$ is said to be locally bounded if given any compact set $K\subset G\times(0,\infty)$, there exists a positive constant $C_K$ such that
for all $j$ and all $x\in K$
\begin{equation*}
|u_j(x)|\leq C_K.
\end{equation*}
\item[v)] A sequence of positive measures $\{\mu_j\}$ on $G$ is said to converge to a positive measure $\mu$ on $G$ in weak* if
\begin{equation*}
\lim_{j\to\infty}\int_{G}\psi(y)\:d\mu_j(y)=\int_{G}\psi(y)\:d\mu(y),\:\:\:\:\text{for all}\:\:\psi\in C_c(G).
\end{equation*}
\end{enumerate}
\end{defn}
\section{An example}
Let us start by briefly discussing about the Heisenberg groups $H^n$. As a set, $H^n$ is $\C^n\times\R$. Denoting the points of $H^n$ by $(z,s)$ with $z=(z_1,\cdots,z_n)\in\C^n,\:s\in\R$, we have the group law given as
	\begin{equation*}
	(z,s)\circ(z',s')=\left(z+z',s+s'+\frac{1}{2}\sum_{j=1}^n\Im(z_j\overline{z_j'})\right)
	\end{equation*}
	With the notation $z_j=x_j+y_j$, $V_1=\R^{2n}\times\{0\}$ is spanned by the basis
	\begin{equation*}
	X_j=\frac{\partial}{\partial x_j}+\frac{1}{2}y_j\frac{\partial}{\partial s},\:\:Y_j=\frac{\partial}{\partial y_j}-\frac{1}{2}x_j\frac{\partial}{\partial s},\:\:\:1\leq j\leq n.
	\end{equation*}
	The one dimensional centre $V_2=\{0\}\times\R$ is generated by the vector field
	\begin{equation*}
	S=\frac{\partial}{\partial s}.
	\end{equation*}
	The nonzero Lie brackets of the basis elements are given by
	\begin{equation*}
	[X_j,Y_j]=-S,\:\:\:\: 1\leq j\leq n.
	\end{equation*}
	The sub-Laplacian $\mathcal{L}=\sum_{j=1}^n(X_j^2+Y_j^2)$ is known as the Kohn Laplacian in the literature. The corresponding homogeneous norm $d_{\mathcal{L}}$ (known as Koranyi norm) on $H^n$ is given by the following expression \cite[P.696]{BLU}:
	\begin{equation*}
	d_{\mathcal{L}}(z,s)=\left(|z|^4+16s^2\right)^{\frac{1}{4}}.
	\end{equation*}
	We refer the reader to \cite{BLU} for more examples of stratified Lie groups.
\begin{defn}
Let $\mu$ be a measure on a stratified Lie group $G$ and $x_0\in G$. We say that $x_0$ is a Lebesgue point of $\mu$ if there exists $L\in\C$, such that
\begin{equation}\label{leblimcarnot}
\lim_{r\to 0}\frac{|\mu-Lm|(B(x_0,r))}{m(B(0,r))}=0.
\end{equation}
\end{defn}
\begin{rem}
If $x_0$ is a Lebesgue point of $\mu$ with $L$ as in (\ref{leblimcarnot}), then the strong derivative (see Definition \ref{impdefnc}, ii)) of $\mu$ at $x_0$ exists and equals $L$. Indeed, we take a ball $B=B(x,t)$ in $G$ and note that
		\begin{eqnarray*}
			\left|\frac{\mu(x_0\circ\delta_r(B))}{m(x_0\circ\delta_r(B))}-L\right|&=&\left|\frac{\mu\left(B(x_0\circ\delta_r(x),rt\right)-Lm\left(B(x_0\circ\delta_r(x),rt\right)}{m(B\left(0,rt)\right)}\right|\\&\leq&\frac{|\mu-Lm|\left(B(x_0\circ\delta_r(x),rt\right)}{m\left(B(0,rt)\right)}\\&\leq&\frac{|\mu-Lm|\left(B(x_0,\tau r(t+d_{\mathcal{L}}(x))\right)}{m\left(B(0,rt)\right)}\\&\leq&\frac{|\mu-Lm|\left(B(x_0,\tau r(t+d_{\mathcal{L}}(x))\right)}{m\left(B_d(0,\tau r(t+d_{\mathcal{L}}(x)))\right)}\times\left(\frac{\tau r(t+d_{\mathcal{L}}(x))}{rt}\right)^Q,
		\end{eqnarray*}
		where $\tau$ is the constant $C_{\mathcal{L}}$ in the inequality (\ref{revtri}). Using (\ref{leblimcarnot}), we see that the right-hand side of the last inequality goes to zero as $r$ goes to zero. As the $d_{\mathcal{L}}$-ball $B$ is arbitrary, $D\mu(x_0)$ is equal to $L$.
\end{rem} 
We now construct an absolutely continuous measures on the Heisenberg group $H^1$ to show that the strong derivative exists outside the set of all Lebesgue points. We note that $H^1=\R\times\R\times\R$ has homogeneous dimension $4$, and the group law is given by
	\begin{equation*}
	(x,y,t)\circ(x',y',t')=\left(x+x',y+y',t+t'+\frac{1}{2}(x'y-xy')\right).
	\end{equation*}
	We will use Shapiro's construction to produce an absolutely continous measure $\mu$ on $H^1$ such that  the strong derivative of $\mu$ at $(0,0,0)$ exists, but $(0,0,0)$ is not a Lebesgue point of $\mu$. 
	Shapiro \cite[P.3185]{Sh} began with constructing an odd function $g:\R\to [-1,1]$, satisfying the following.
	\begin{enumerate}
		\item [i)] $g$ is continuous everywhere except at $0$, with $g(0)=0$.
		\item[ii)] For all $s\in(0,1]$,
		\begin{equation}\label{notlebesgue}
		s^{-1}\int_{0}^{s}|g(t)|\:dt\geq\frac{1}{6}.
		\end{equation}
	\end{enumerate}
	Shapiro then considered the definite integral $G$ of $g$
	\begin{equation*}
	G(s)=\int_{0}^sg(t)\:dt,\:\:\:\:s\in\R,
	\end{equation*}
	and proved that $G$ has the following properties.
	\begin{enumerate}
		\item [i)] $G$ is even, differentiable everywhere and 
		\begin{equation}\label{derivativeofG}
		G'(s)=g(s),\:\:\:\:\text{for all}\:\:\:s\in\R.	
		\end{equation}
		\item[ii)] For all $s$ with $|s|\in (0,1]$,
		\begin{equation}\label{estimateofG}
		\frac{|G(s)|}{|s|}\leq |s|.
		\end{equation}
	\end{enumerate}
	We now define the function $f:H^1\to [-1,1]$, for our example as follows:
	\begin{equation*}
	f(x,y,t)=
	\begin{cases}
	g(x),\:\:\:\:\text{for}\:\:\:\:d_{\mathcal{L}}(x,y,t)\leq 10\\0,\:\:\:\:\text{for}\:\:\:\:d_{\mathcal{L}}(x,y,t)>10.
	\end{cases}
	\end{equation*}
	It is clear that $f\in L^p(H^1)$, for any $p\in[1,\infty]$. For $r\in(0,1)$, we define 
	\begin{equation*}
	\mathrm{Q}(r)=\{(x,y,t)\in H^1\mid |x|<r,\:|y|<r,\:|t|<r^2\}.
	\end{equation*}
	It is evident that $\mathrm{Q}(r)\subset B((0,0,0),20^{\frac{1}{4}}r)$. Therefore, we get for $r\in(0,1)$ 
	\begin{eqnarray}\label{lebesgueest}
		r^{-4}\int_{B((0,0,0),20^{\frac{1}{4}}r)}|f(x,y,t)|\:dxdydt&\geq& r^{-4}\int_{\mathrm{Q}(r)}|f(x,y,t)|\:dxdydt\nonumber\\&=&r^{-4}\int_{-r^2}^{r^2}\int_{-r}^{r}\int_{-r}^{r}|g(x)|\:dxdydt\nonumber\\&=&4r^{-1}\int_{-r}^{r}|g(x)|\:dx\nonumber\\&=&8r^{-1}\int_{0}^{r}|g(x)|\:dx\nonumber\\&\geq&\frac{4}{3},
	\end{eqnarray}
	where the last inequality follows from (\ref{notlebesgue}). This shows that $(0,0,0)$ is not a Lebesgue point of $f$, as $f(0,0,0)=0$. For the second part we need the following version of divergence theorem.
	\begin{lem}[{{\cite[Corollary 7.4]{Pf}}}]\label{divergence}
		Let $\mathrm{M}$ be an $k$-dimensional compact oriented manifold, and let $\omega$ be a continuous $(k-1)$-form on $\mathrm{M}$ which is differentiable in $\mathrm{M}-\partial\mathrm{M}$. Then $d\omega$ is integrable on $\mathrm{M}$ and \begin{equation*}
		\int_{\mathrm{M}}d\omega=\int_{\partial\mathrm{M}}\omega.
		\end{equation*}
	\end{lem}
	We shall apply this version of divergence theorem on the following manifolds.
	\begin{equation*}
	\overline{B\left((x,y,t),r\right)}=\{(u,v,s)\in H^1\mid d_{\mathcal{L}}\left((x,y,t)^{-1}\circ(u,v,s)\right)\leq r\}.
	\end{equation*} 
	We define a function $F:H^1\to\R$, as follows: 
	\begin{equation}\label{capfmeanscapg}
	F(x,y,t)=G(x).
	\end{equation}
	It follows from (\ref{derivativeofG}) that $F$ has a total derivative at each point of $H^1$. Moreover, 
	\begin{eqnarray}\label{partialderivativeofF}
	&&\frac{\partial F}{\partial x}(x,y,t)=G'(x)=g(x)=f(x,y,t),\\&&\frac{\partial F}{\partial y}(x,y,t)=0,\nonumber\\&&\frac{\partial F}{\partial t}(x,y,t)=0,\nonumber.
	\end{eqnarray}
	whenever $d_{\mathcal{L}}(x,y,t)\leq 2$. We note that 
	\begin{equation*}
	\partial\overline{B\left((x,y,t),r\right)}=\{(u,v,s)\in H^1\mid h_{(x,y,t)}(u,v,s)=0\},
	\end{equation*}
	where for $(u,v,s)\in H^1$,
	\begin{equation*}
	h_{(x,y,t)}(u,v,s)=\left((u-x)^2+(v-y)^2\right)^2+16\left(s-t+\frac{1}{2}(xv-yu)\right)^2-r^4.
	\end{equation*}
	We have 
	\begin{eqnarray*}
		&&\frac{\partial h_{(x,y,t)}}{\partial u}(u,v,s)=4\left((u-x)^2+(v-y)^2\right)(u-x)-16y\left(s-t+\frac{1}{2}(xv-yu)\right);\\
		&&\frac{\partial h_{(x,y,t)}}{\partial v}(u,v,s)=4\left((u-x)^2+(v-y)^2\right)(v-y)+16x\left(s-t+\frac{1}{2}(xv-yu)\right);\\
		&&\frac{\partial h_{(x,y,t)}}{\partial s}(u,v,s)=32\left(s-t+\frac{1}{2}(xv-yu)\right).
	\end{eqnarray*}
	It is clear that the partial derivatives of $h_{(x,y,t)}$ can not be simultaneously vanishing on $\partial \overline{B\left((x,y,t),r\right)}$. 
	Thus, applying Lemma \ref{divergence}, we obtain for $d(x,y,t)<1$, $r<1$,
	\begin{equation*}
	\int_{\overline{B\left((x,y,t),r\right)}}\text{div}\:(F,0,0)\:dm=\int_{\partial\overline{B\left((x,y,t),r\right)}}(F,0,0)\cdot n\:dS,
	\end{equation*}
	where $n$ is the outward unit normal to the surface $\partial\overline{B\left((x,y,t),r\right)}$ and $dS$ is the surface measure on $\partial\overline{B\left((x,y,t),r\right)}$. Using (\ref{partialderivativeofF}), we obtain from the equation above that 
	\begin{eqnarray}\label{surfaceintegral}
	\int_{B((x,y,t),r)}f(x,y,t)\:dxdydt&=&\int_{\partial \overline{B\left((x,y,t),r\right)}}(F,0,0)\cdot n\:dS.
	\end{eqnarray}
	We note that 
	\begin{equation*}
	\partial \overline{B\left((x,y,t),r\right)}=(x,y,t)\circ\partial\overline{B\left((0,0,0),r\right)}.
	\end{equation*}
	We have the following parametrization of $\partial\overline{B\left((0,0,0),r\right)}$ (see \cite[P.133]{Garo}).
	\begin{equation*}
	\partial\overline{B\left((0,0,0),r\right)}=\{\Psi(\phi,\theta)\mid\phi\in(0,\pi),\:\theta\in(0,2\pi)\},
	\end{equation*} where
	\begin{equation*}
	\Psi(\phi,\theta)=(r\sqrt{\sin\phi}\sin\theta,r\sqrt{\sin\phi}\cos\theta,r^2\cos\phi).
	\end{equation*}
	Using this we get the following parametrization of $\partial\overline{B\left((x,y,t),r\right)}$. 
	\begin{equation*}
	\partial\overline{B_d\left((x,y,t),r\right)}=\{\Psi_{(x,y,t)}(\phi,\theta)\mid\phi\in(0,\pi),\:\theta\in(0,2\pi)\},
	\end{equation*}where for $\phi\in(0,\pi)$, $\theta\in(0,2\pi)$
	\begin{eqnarray*}
		&&\Psi_{(x,y,t)}(\phi,\theta)\\&=&\left(x+r\sqrt{\sin\phi}\sin\theta,y+r\sqrt{\sin\phi}\cos\theta,t+r^2\cos\phi-\frac{r}{2}\sqrt{\sin\phi}(x\cos\theta-y\sin\theta)\right).
	\end{eqnarray*}
	Therefore, 
	\begin{eqnarray*}
		&&\frac{\partial\Psi_{(x,y,t)}}{\partial \phi}(\phi,\theta)\\&=&\left(r\frac{\cos\phi}{2\sqrt{\sin\phi}}\sin\theta,r\frac{\cos\phi}{2\sqrt{\sin\phi}}\cos\theta,-r^2\sin\phi-\frac{r\cos\phi}{4\sqrt{\sin\phi}}(x\cos\theta-y\sin\theta)\right);
	\end{eqnarray*}
	and 
	\begin{equation*}
	\frac{\partial\Psi_{(x,y,t)}}{\partial \theta}(\phi,\theta)=\left(r\sqrt{\sin\phi}\cos\theta,-r\sqrt{\sin\phi}\sin\theta,\frac{r}{2}\sqrt{\sin\phi}(x\sin\theta+y\cos\theta)\right).
	\end{equation*}
	To evaluate the right-hand side of (\ref{surfaceintegral}), we need only the first coordinate of 
	\begin{eqnarray*}
		&&\frac{\partial\Psi_{(x,y,t)}}{\partial \phi}\times\frac{\partial\Psi_{(x,y,t)}}{\partial \theta}(\phi,\theta)\\&=&\begin{vmatrix}
			i & j & k\\r\frac{\cos\phi}{2\sqrt{\sin\phi}}\sin\theta & r\frac{\cos\phi}{2\sqrt{\sin\phi}}\cos\theta & -r^2\sin\phi-\frac{r\cos\phi}{4\sqrt{\sin\phi}}(x\cos\theta-y\sin\theta)\\r\sqrt{\sin\phi}\cos\theta & -r\sqrt{\sin\phi}\sin\theta & \frac{r}{2}\sqrt{\sin\phi}(x\sin\theta+y\cos\theta)
		\end{vmatrix},
	\end{eqnarray*}
	which is equal to 
	\begin{equation*}
	\frac{r^2}{4}y\cos\phi-r^3\sin^{\frac{3}{2}}\phi\sin\theta.
	\end{equation*}
	Using this, together with the definition of $F$ (see (\ref{capfmeanscapg})) in (\ref{surfaceintegral}), we obtain from (\ref{surfaceintegral}) that
	\begin{eqnarray*}
		&&\int_{B_d((x,y,t),r)}f(x,y,t)\:dxdydt\\&=&\int_{0}^{2\pi}\int_{0}^{\pi}G(x+r\sqrt{\sin\phi}\sin\theta)\left(\frac{r^2}{4}y\cos\phi-r^3\sin^{\frac{3}{2}}\phi\sin\theta\right)\:d\phi\:d\theta.
	\end{eqnarray*}
	As $d_{\mathcal{L}}(x,y,t)$ is bigger than $|x|$, we have
	\begin{equation*}
	|x+r\sqrt{\sin\phi}\sin\theta|\leq d_{\mathcal{L}}(x,y,t)+r.
	\end{equation*}
	Hence, by the estimate (\ref{estimateofG}), we get for all $(x,y,t)\in H^1$, $r>0$,  with  $d_{\mathcal{L}}(x,y,t)+r\in(0,1)$, that
	\begin{eqnarray*}
		&&\left|\int_{B((x,y,t),r)}f(x,y,t)\:dxdydt\right|\\&\leq&\int_{0}^{2\pi}\int_{0}^{\pi}|x+r\sqrt{\sin\phi}\sin\theta|^2\left|\frac{r^2}{4}y\cos\phi-r^3\sin^{\frac{3}{2}}\phi\sin\theta\right|\:d\phi\:d\theta\\&\leq&2\pi^2(d_{\mathcal{L}}(x,y,t)+r)^2\left(\frac{r^2}{4}|y|+r^3\right)\\&\leq&2\pi^2(d_{\mathcal{L}}(x,y,t)+r)^2\left(r^2d_{\mathcal{L}}(x,y,t)+r^3\right)\\&=&2\pi^2r^2(d_{\mathcal{L}}(x,y,t)+r)^3\\&\leq&2\pi^2r(d_{\mathcal{L}}(x,y,t)+r)^4.
	\end{eqnarray*}
	Thus, for a given $\epsilon>0$, choosing $\eta=\min\{\frac{\epsilon}{2\pi^2},1\}$, yields
	\begin{equation}\label{sigmaest}
	\left|\int_{B((x,y,t),r)}f(x,y,t)\:dxdydt\right|\leq \epsilon(d_{\mathcal{L}}(x,y,t)+r)^4,
	\end{equation}
	whenever $0<(d_{\mathcal{L}}(x,y,t)+r)<\eta$. Set $d\mu=f\:dm$. Fix a ball $B=B(P,t)$. Using (\ref{sigmaest}) we obtain that
	\begin{eqnarray*}
	\left|\frac{\mu(rB)}{m(rB)}\right|&=&m(B(0,1))(rt)^{-4}\left|\int_{B(\delta_r(P),rt)}f(x,y,t)\:dxdydt\right|\\&\leq&m(B(0,1))(rt)^{-4}\epsilon(d_{\mathcal{L}}\left(\delta_r(P)\right)+rt)^4\\&=&m(B(0,1))t^{-4}\left(d_{\mathcal{L}}(P)+t\right)^4\epsilon,
	\end{eqnarray*}
	whenever $0<d_{\mathcal{L}}(\delta_r(P))<\eta$, and $rt<\eta$. Taking $r_0=\min\{\frac{\eta}{d_{\mathcal{L}}(P)+1},\frac{\eta}{s}\}$, we obtain from the last inequality that
	\begin{equation*}
	\left|\frac{\mu(rB)}{m(rB)}\right|\leq m(B(0,1))t^{-4}\left(d_{\mathcal{L}}(P)+t\right)^4\epsilon,
	\end{equation*}
	 for all $r\in(0,r_0)$. This shows that $\mu$ has strong derivative zero at $(0,0,0)$. On the other hand (\ref{lebesgueest}) shows that $(0,0,0)$ is not a Lebesgue point.
	Since $|f|$ is bounded by one, we have 
	\begin{equation*}
	f_1=(f+1)\chi_{B\left((0,0,0),10\right)}\geq 0,
	\end{equation*}
	and $f_1\in L^p(H^1)$, for any $p\in[1,\infty]$. It now follows by setting $d\nu=f_1\:dm$ that $\nu$ is a positive measure on $H^1$ such that the strong derivative exists and equals one at $(0,0,0)$ but $(0,0,0)$ is not a Lebesgue point of $\nu$.
\section{Some auxilary results}
We start this section with the following result involving normal convergence and weak* convergence.
\begin{lem}\label{normalc}
Suppose $\{\mu_j\mid j\in\N\}\subset M$ and $\mu\in M$ are positive measures. If $\{\Gamma\mu_j\}$ converges normally to $\Gamma\mu$ then $\{\mu_j\}$ converges to $\mu$ in weak*.
\end{lem}
\begin{proof}
Let $f\in C_c(G)$ with $\text{supp}\:f\subset B(0,R)$ for some $R>0$. For any $t>0$, we write
\begin{eqnarray}
&&\int_{G}f(x)\:d\mu_j(x)-\int_{G}f(x)\:d\mu(x)\nonumber\\
&=&\int_{G}(f(x)-\Gamma f(x,t))\:d\mu_j(x)+\int_{G}\Gamma f(x,t)\:d\mu_j(x)-\int_{G}\Gamma f(x,t)\:d\mu(x)\nonumber\\
&&\:\:\:\:\:+\int_{G}(\Gamma f(x,t)-f(x))\:d\mu(x).\label{ineq1c}
\end{eqnarray}
Given $\epsilon>0$, by Proposition \ref{unifc} we get some $t_0>0$, such that for all $x\in G$
\begin{equation}\label{uniformineq}
\frac{|\Gamma f(x,t_0)-f(x)|}{\gamma(x)}<\epsilon.
\end{equation}
Using Proposition \ref{fubinic}, it follows from (\ref{ineq1c}) that
\begin{eqnarray*}
&&\left|\int_{G}f(x)\:d\mu_j(x)-\int_{G}f(x)\:d\mu(x)\right|\\
&\leq &\int_{G}|f(x)-\Gamma f(x,t_0)|\:d\mu_j(x)+\int_{B(0,R)}|\Gamma\mu_j(x,t_0)-\Gamma\mu(x,t_0)||f(x)|\:dm(x)\\
&&\:\:\:\:\:+\int_{G}|\Gamma f(x,t_0)-f(x)|\:d\mu(x)\\
&=&I_1(j)+I_2(j)+I_3.
\end{eqnarray*}
It follows from (\ref{uniformineq}) that
\begin{equation*}
I_1(j)=\int_{G}\frac{|\Gamma f(x,t_0)-f(x)|}{\gamma(x)}\gamma(x)\:d\mu_j(x)\leq \epsilon\int_{G}\gamma(x)\:d\mu_j(x)=\epsilon \Gamma\mu_j(0,1),
\end{equation*}
for all $j\in\N$. By the same argument, we also have that 
\begin{equation*}
|I_3|\leq\epsilon \Gamma\mu(0,1).
\end{equation*}
Since $\{\Gamma\mu_j\}$ converges to $\Gamma\mu$ normally, the sequence $\{\Gamma\mu_j(0,1)\}$, in particular, is bounded. Hence, taking $A$ to be the supremum of $\{\Gamma\mu_j(0,1)+\Gamma\mu(0,1)\}$, we get that for all $j\in\N$
\begin{equation*}
I_1(j)+I_3\leq2A\epsilon.
\end{equation*}
Since $\{\Gamma\mu_j\}$ converges normally to $\Gamma\mu$, there exists $j_0\in\N$ such that for all $j\geq j_0$,
\begin{equation*}
\|\Gamma\mu_j-\Gamma\mu\|_{L^{\infty}(\overline{B(0,R)}\times\{t_0\})}<\epsilon.
\end{equation*}
This implies that for all $j\geq j_0$,
\begin{equation*}
I_2(j)\leq \epsilon \|f\|_{L^1(G)}.
\end{equation*}
Hence, for all $j\geq j_0$,
\begin{equation*}
\left|\int_{G}f(x)d\mu_j(x)-\int_{G}f(x)d\mu(x)\right|\leq\epsilon (2A+\|f\|_{L^1(G)}).
\end{equation*}
This completes the prove.
\end{proof}
It is well-known that if two positive measures on $\R^n$ agree on all open balls, then they are equal. We are now going to prove that the same conclusion can be drawn when open balls are replaced by $d_{\mathcal{L}}$-balls. 
\begin{prop}\label{measuresequal}
Let $\mu$ and $\nu$ be two positive measures on $G$. If 
\begin{equation}\label{equalonball}
\mu(B)=\nu(B),
\end{equation}
for every $d_{\mathcal{L}}$-ball $B\subset G$, then $\mu=\nu$.
\end{prop}
\begin{proof}
We set 
\begin{equation*}
\phi=m\left(B(0,1)\right)^{-1}\chi_{B(0,1)}.
\end{equation*}
Since translation and dilation of a $d_{\mathcal{L}}$-ball is again a $d_{\mathcal{L}}$-ball, it follows that for all $x\in G$ and $r>0$,
\begin{equation}
\mu\ast\phi_r(x)=\nu\ast\phi_r(x)\label{muphiequalnuphi}.
\end{equation}
It follows from \cite[Theorem 2.18]{Rub} that $\mu$, $\nu$ are regular and hence it suffices to show that 
\begin{equation*}
\int_{G}g\:d\mu=\int_{G}g\:d\nu,\:\:\:\:\:\text{for all}\:\:g\in C_c(G).
\end{equation*}
We take $f\in C_c(G)$ with $\text{supp}f\subset B(0,R)$. We consider for $x\in G$, $r>0$, 
\begin{eqnarray}\label{fmuphi}
f\ast(\mu\ast\phi_r)(x)&=&\int_{G}f(y)\mu\ast\phi_r(y^{-1}\circ x)\:dm(y)\nonumber\\&=&\int_{G}f(y)\int_{G}\phi_r\left(\xi^{-1}\circ (y^{-1}\circ x)\right)\:d\mu(\xi)\:dm(y)\nonumber\\
&=&\int_{G}\int_{G}f(y_1\circ\xi^{-1})\phi_r(y_1^{-1}\circ x)\:d\mu(\xi)\:dm(y_1)\nonumber\\
&&\:\:\:\:(\text{substituting}\:\:y=y_1\circ\xi^{-1}\:\:\text{and using the translation invariance of}\:m)\nonumber\\&=&\int_{G}f_{\mu}(y_1)\phi_r(y_1^{-1}\circ x)\:dm(y_1)\nonumber\\&=&f_{\mu}\ast\phi_r(x),
\end{eqnarray} where
\begin{equation}\label{fmudefinition}
f_{\mu}(y)=\int_{G}f(y\circ\xi^{-1})\:d\mu(\xi),\:\:\:y\in G.
\end{equation}
We now claim that $f_{\mu}$ is continuous at $0$. To see this, we consider a sequence $\{y_k\}$ converging to $0$. Since the group operation and $d_{\mathcal{L}}$ are continuous, $y_k\circ\xi^{-1}\to\xi^{-1}$, for each $\xi\in G$, and there exists some positive constant $A$ such that $d_{\mathcal{L}}(y_k)\leq A$, for all $k$. Note that for $d_{\mathcal{L}}(\xi)>C_{\mathcal{L}}(R+A)$,
\begin{equation*}
d_{\mathcal{L}}(y_k\circ\xi^{-1})\geq \frac{1}{C_{\mathcal{L}}}d_{\mathcal{L}}(\xi)-d_{\mathcal{L}}(y_k)>\frac{1}{C_{\mathcal{L}}}C_{\mathcal{L}}(R+A)-A=R, \:\:\:\:\text{for all}\:\:k.
\end{equation*}
Therefore, we can write for all $k$,
\begin{equation}\label{anotherfmu}
f_{\mu}(y_k)=\int_{B(0,C_{\mathcal{L}}(R+A))}f(y_k\circ\xi^{-1})\:d\mu(\xi).
\end{equation}
By continuty of $f$,  $f(y_k\circ\xi^{-1})\to f(\xi^{-1})$, for each $\xi$, and hence applying dominated convergence theorem on the righ-hand side of (\ref{anotherfmu}), we obtain
\begin{equation*}
f_{\mu}(y_k)\to\int_{B(0,C_{\mathcal{L}}(R+A))}f(\xi^{-1})\:d\mu(\xi)=\int_{G}f(\xi^{-1})\:d\mu(\xi)=f_{\mu}(0),\:\:\:\:\text{as}\:\:k\to\infty.
\end{equation*}
This proves our claim. Let $\epsilon>0$. Using (\ref{bilipschitz}) we choose some $\delta>0$, such that
\begin{equation*}
|f_{\mu}(y)-f_{\mu}(0)|<\epsilon,\:\:\:\:\text{for all}\:\:\:y\in B(0,\delta).
\end{equation*}
Hence, \begin{eqnarray*}
|f_{\mu}\ast\phi_r(0)-f_{\mu}(0)|&=&\left|\int_{G}f_{\mu}(\xi)\phi_r(\xi^{-1})\:dm(\xi)-\int_{G}f_{\mu}(0)\phi_r(\xi^{-1})\:dm(\xi)\right|\\&\leq&\frac{1}{m(B(0,r))}\int_{B(0,r)}|f_{\mu}(\xi)-f_{\mu}(0)|\:dm(\xi)\\&<&\epsilon,\:\:\:\:\text{for all}\:\:0<r<\delta.
\end{eqnarray*}
This together with (\ref{fmuphi}) and (\ref{fmudefinition}), implies that \begin{equation*}
f\ast(\mu\ast\phi_r)(0)\to f_{\mu}(0)=\int_{G}f(\xi^{-1})\:d\mu(\xi),  \:\:\:\:\text{as}\:\:r\to 0.
\end{equation*} 
Similarly, we can prove that \begin{equation*}
f\ast(\nu\ast\phi_r)(0)\to f_{\nu}(0)=\int_{G}f(\xi^{-1})\:d\nu(\xi),  \:\:\:\:\text{as}\:\:r\to 0,
\end{equation*}
where $f_{\nu}$ is defined according to (\ref{fmudefinition}). Equation (\ref{muphiequalnuphi}) now shows that \begin{equation*}
\int_{G}f(\xi^{-1})\:d\mu(\xi)=\int_{G}f(\xi^{-1})\:d\nu(\xi).
\end{equation*}
This completes the proof.
\end{proof}
We now use this proposition to prove the following measure theoretic result that will be needed in the proof of our main theorem.
\begin{lem}\label{mthc}
Suppose $\{\mu_j\}_{j\geq 1}$, $\mu$  are positive measures on $G$ and $\{\mu_j\}$ converges to $\mu$ in weak*. Then for some $L\in[0,\infty)$, $\mu=Lm$ if and only if $\{\mu_j(B)\}$ converges to $Lm(B)$ for every $d_{\mathcal{L}}$-ball $B\subset G$.
\end{lem}
\begin{proof}
Suppose $\mu=Lm$. Fix a $d_{\mathcal{L}}$-ball $B\subset G$ and $\epsilon>0$. As $\overline{B}$ is compact with respect to the Euclidean topology, by regularity of the Lebesgue measure $m$, there exists an open set $V\supset\overline{B}$ such that $m(V\setminus\overline{B})<\epsilon$. Using Uryshon's lemma \cite[Theorem 2.12]{Rub}, we choose $\psi\in C_c(G)$ such that 
\begin{equation*}
0\leq\psi(x)\leq 1,\:\:\text{for all}\:x\in G;\: \psi\equiv1\:\text{on}\:\overline{B};\:\:\psi\equiv0\:\text{on}\:G\setminus V. 
\end{equation*}
Then
\begin{equation}\label{regular1}
\int_{G}\psi\:dm=\int_{B}\psi\:dm+\int_{V\setminus\overline{B}}\psi\:dm\leq m(B)+m(V\setminus\overline{B})\leq m(B)+\epsilon.
\end{equation}
As $\psi\equiv1\:\text{on}\:\overline{B}$ and $\mu_j\to\mu$ in weak*,
\begin{equation*}
\limsup_{j\to\infty}\mu_j(B)=\limsup_{j\to\infty}\int_{B}\psi\:d\mu_j\leq\limsup_{j\to\infty}\int_{G}\psi\:d\mu_j=\int_{G}\psi\:d\mu.
\end{equation*}
Using our assumption, that is, $\mu=Lm$ and (\ref{regular1}) in the above, we get
\begin{equation*}
\limsup_{j\to\infty}\mu_j(B)\leq L\int_{G}\psi\:dm\leq L(m(B)+\epsilon)
\end{equation*}
 Since $\epsilon>0$ is arbitrary
\begin{equation}\label{limsup}
\limsup_{j\to\infty}\mu_j(B)\leq Lm(B).
\end{equation}
Similarly, by choosing a compact set $K\subset B$ with
\begin{equation*}
 m(K)>m(B)-\epsilon\:\:\:\:\:(\text{using  Remark \ref{topology}})
\end{equation*}
 and a function $g\in C_c(G)$ such that 
\begin{equation*}
0\leq g(x)\leq 1,\:\:\text{for all}\:x\in G;\: g\equiv1\:\text{on}\:K;\:\:g\equiv0\:\text{on}\:G\setminus B,
\end{equation*} we observe that
\begin{equation*}
\int_{G}g\:dm\geq \int_{K}g\:dm=m(K)>m(B)-\epsilon.
\end{equation*}
As $0\leq g\leq1$ with $\text{supp}\:g\subset B$ and $\mu_j\to\mu$ in weak*,
\begin{equation*}
\liminf_{j\to\infty}\mu_j(B)\geq\liminf_{j\to\infty}\int_{G}g\:d\mu_j=\int_{G}g\:d\mu=L\int_{G}g\:dm>L(m(B)-\epsilon).
\end{equation*}
Since $\epsilon>0$ is arbitrary 
\begin{equation*}
\liminf_{j\to\infty}\mu_j(B)\geq Lm(B).
\end{equation*}
Combining the above inequality with (\ref{limsup}) we conclude that
\begin{equation*}
\lim_{j\to\infty}\mu_j(B)=Lm(B).
\end{equation*}
Conversely, we suppose that
\begin{equation}\label{limmuj}
\lim_{j\to\infty}\mu_j(B)=Lm(B),
\end{equation} 
for every $d_{\mathcal{L}}$-ball $B\subset G$. We need to prove that $\mu=Lm$. In view of Proposition \ref{measuresequal}, it suffices to show that $\mu(B)=Lm(B)$, for every $d_{\mathcal{L}}$-ball $B\subset G$. The proof of this part is similar to that of the previous part. We fix $\epsilon>0$ and a $d_{\mathcal{L}}$-ball $B=B(x_0,r)\subset G$. We denote the $d_{\mathcal{L}}$-ball centred at $x_0$ and radius $r+\epsilon$ by $B^{\prime}$. Taking Remark \ref{compact} into account and applying Uryshon's lemma we get a function $f\in C_c(G)$ such that 
\begin{equation*}
0\leq f(x)\leq 1,\:\:\text{for all}\:x\in G;\: f\equiv1\:\text{on}\:\overline{B};\:\:f\equiv0\:\text{on}\:G\setminus B^{\prime}.
\end{equation*}
Using our hypothesis, namely $\mu_j\to\mu$ in weak*, the above implies that 
\begin{equation*}
\mu(B)=\int_{B}f\:d\mu\leq\int_{G}f\:d\mu=\lim_{j\to\infty}\int_{G}f\:d\mu_j\leq\lim_{j\to\infty}\mu_j(B^{\prime})=Lm(B^{\prime})=Lm(B(0,1))(r+\epsilon)^Q.
\end{equation*}
Since $\epsilon>0$ is arbitrary,
\begin{equation*}
\mu(B)\leq Lm(B(0,1))r^Q=Lm(B).
\end{equation*} 
Similarly, letting $B^{\prime\prime}=B(x_0,r-\epsilon)$ and choosing a function $f_1\in C_c(G)$ such that 
\begin{equation*}
0\leq f_1(x)\leq 1,\:\:\text{for all}\:x\in G;\: f_1\equiv1\:\text{on}\:B^{\prime\prime};\:\:f_1\equiv0\:\text{on}\:G\setminus B,
\end{equation*}
we obtain
\begin{equation*}
\mu(B)\geq\int_{G}f_1\:d\mu=\lim_{j\to\infty}\int_{G}f_1\:d\mu_j\geq\liminf_{j\to\infty}\int_{B^{\prime\prime}}f_1\:d\mu_j=\liminf_{j\to\infty}\mu_j(B^{\prime\prime}).
\end{equation*}
Consequently, (\ref{limmuj}) gives
\begin{equation*}
\mu(B)\geq Lm(B^{\prime\prime})=Lm(B(0,1))(r-\epsilon)^Q.
\end{equation*}
As $\epsilon>0$ is arbitrary, 
\begin{equation*}
\mu(B)\geq Lm(B).
\end{equation*}
This completes the proof.
\end{proof}
Next, we shall consider various types of maximal functions on $G$. For a measureable function $\phi$ defined on $G$ and a complex or a signed measure $\mu$, we define the $\alpha$-nontangential maximal function $M_{\phi}^{\alpha}\mu$, where $\alpha>0$ and the radial maximal function $M_{\phi}^0\mu$ of $\mu$ with respect to $\phi$ 
\begin{eqnarray*}
&&M_{\phi}^{\alpha}\mu(x)=\sup_{\substack{(\xi,t)\in G\times(0,\infty)\\d_{\mathcal{L}}(x,\xi)<\alpha t}}|\mu\ast\phi_t(\xi)|,\:x\in G,\\&&M_{\phi}^0\mu(x)=\sup_{0< t<\infty}|\mu\ast\phi_t(\xi)|,\:x\in G.
\end{eqnarray*}
It is obvious that $M_{\phi}^0\mu$ is pointwise dominated by $M_{\phi}^{\alpha}\mu$ for all $\alpha>0$. In \cite[Corollary 2.5]{FS}, it was proved that if $\phi$ satisfies some polynomial decay, namely
\begin{equation*}
|\phi(x)|\leq A(1+d_{\mathcal{L}}(x))^{-\lambda}
\end{equation*}
for some $A>0$ and $\lambda>Q$, then $M_{\phi}^{\alpha}:L^1(G)\to L^{1,\infty}(G)$ and $M_{\phi}^{\alpha}:L^p(G)\to L^{p}(G)$, $1< p\leq\infty$. Although Folland-Stein proved these mapping properties of $M_{\phi}^{\alpha}$ for $\alpha=1$ but their proof works for all $\alpha>0$. An important special case of this type of maximal functions is the centred Hardy-Littlewood maximal function, which is obtained by taking $\phi=\chi_{B(0,1)}$ in $M_{\phi}^0\mu$. We shall denote it by $M_{HL}(\mu)$. In other words,
\begin{equation*}
M_{HL}(\mu)(x)=\sup_{r>0}\frac{|\mu(B(x,r))|}{m(B(x,r))},\:\:\:\:x\in G.
\end{equation*}
In the following, we shall prove a lemma regarding pointwise comparison between the centred Hardy-Littlewood maximal function and other maximal functions introduced above. We then use it to prove the corresponding result for heat maximal functions. 
\begin{lem}\label{maximal}
Let $\phi:G\to(0,\infty)$ be a $\mathcal{L}$-radial, $\mathcal{L}$-radially decreasing (see (\ref{lradial}), (\ref{ldec})) and integrable function. If $\mu$ is a positive measure on $G$ and $\alpha>0$ then there exist positive constants $c_{\alpha,\phi}$ and $c_{\phi}$ such that
\begin{equation*}
c_{\phi}M_{HL}(\mu)(x_0)\leq M_{\phi}^0\mu(x_0)\leq M_{\phi}^{\alpha}\mu(x_0)\leq c_{\alpha,\phi}M_{HL}(\mu)(x_0),
\end{equation*}  
for all $x_0\in G$. The constants $c_{\phi}$ and $c_{\alpha,\phi}$ are independent of $x_0$.
\end{lem}
\begin{proof}
We have already observed that the second inequality is obvious. To prove the left-most inequality we take $t>0$ and note that
\begin{eqnarray*}
\mu\ast\phi_t(x_0)&\geq&\int_{B(x_0,t)}\phi_t(\xi^{-1}\circ x_0)\:d\mu(\xi)\\
&=&t^{-Q}\int_{B(x_0,t)}\phi\left(\delta_{\frac{1}{t}}(\xi^{-1}\circ x_0)\right)\:d\mu(\xi)\\
&\geq&t^{-Q}\int_{B(x_0,t)}\phi(1)\:d\mu(\xi)\:\:\:\left(\text{as}\:\:d_{\mathcal{L}}\left(\delta_{\frac{1}{t}}(\xi^{-1}\circ x_0\right)<1\right)\\
&=&\phi(1)m(B(0,1))\frac{\mu(B(x_0,t))}{m(B(x_0,t))}.
\end{eqnarray*}
Taking supremum over $t$ on both sides we get
\begin{equation}\label{phiradial}
c_{\phi}M_{HL}(\mu)(x_0)\leq M_{\phi}^0\mu(x_0),
\end{equation}
where $c_{\phi}=\phi(1)m(B(0,1))$. To prove the right-most inequality, we take $(\xi,t)\in G\times(0,\infty)$ such that $d_{\mathcal{L}}(x_0,\xi)<\alpha t$. Then,
\begin{eqnarray}
\nonumber\mu\ast\phi_t(\xi)&=&\int_{G}\phi_t(x^{-1}\circ \xi)\:d\mu(x)\\\nonumber
&=&t^{-Q}\int_{\{x:d_{\mathcal{L}}(x^{-1}\circ \xi)<\alpha t\}}\phi\left(\delta_{\frac{1}{t}}(x^{-1}\circ\xi)\right)\:d\mu(x)\\\nonumber
&&\:\:\:\:\:+t^{-Q}\sum_{j=1}^{\infty}\int_{\{x:2^{j-1}\alpha t\leq d_{\mathcal{L}}(x^{-1}\circ \xi)<2^j\alpha t\}}\phi\left(\delta_{\frac{1}{t}}(x^{-1}\circ\xi)\right)\:d\mu(x)\\\nonumber
&\leq&\phi(0)t^{-Q}\mu(B(\xi,\alpha t))+t^{-Q}\sum_{j=1}^{\infty}\int_{\{x:d_{\mathcal{L}}(x^{-1}\circ \xi)<2^j\alpha t\}}\phi(2^{j-1}\alpha)\:d\mu(x)\\\label{ntmax}
&=&\phi(0)t^{-Q}\mu(B(\xi,\alpha t))+t^{-Q}\sum_{j=1}^{\infty}\phi(2^{j-1}\alpha)\mu(B(\xi,2^j\alpha t))
\end{eqnarray}
By the triangle inequality,
\begin{equation*}
d_{\mathcal{L}}(x,x_0)\leq C_{\mathcal{L}}\left(d_{\mathcal{L}}(\xi,x)+d_{\mathcal{L}}(x_0,\xi)\right)\leq C_{\mathcal{L}}(\alpha t+\alpha t)=2C_{\mathcal{L}}\alpha t,
\end{equation*}
whenever $x\in B(\xi,\alpha t)$. Consequently, $B(\xi,\alpha t)\subset B(x_0,2C_{\mathcal{L}}\alpha t)$ and hence
\begin{equation}\label{mumonotone1}
\mu(B(\xi,\alpha t))\leq \mu(B(x_0,2C_{\mathcal{L}}\alpha t)).
\end{equation}
Similarly,
\begin{equation}\label{mumonotone2}
\mu(B(\xi,2^j\alpha t))\leq \mu(B(x_0,C_{\mathcal{L}}(2^j+1)\alpha t)).
\end{equation}
We now use the formula for integration in ``polar coordinates" given in (\ref{polarcordinate}) to get
\begin{equation}\label{polar}
\int_{G}\phi(x)\:dm(x)=\sigma(S)\int_{0}^{\infty}\phi(r)r^{Q-1}\:dr\:\:\:\:\:\:(\text{as $\phi$ is $\mathcal{L}$-radial}).
\end{equation} 
But $\phi$ is $\mathcal{L}$-radially decreasing and nonnegative. Therefore,
 \begin{eqnarray*}
\int_{\alpha}^{\infty}\phi(r)r^{Q-1}\:dr&=&\sum_{j=1}^{\infty}\int_{2^{j-1}\alpha}^{2^j\alpha}\phi(r)r^{Q-1}\:dr\geq\sum_{j=1}^{\infty}\phi(2^j\alpha)\int_{2^{j-1}\alpha}^{2^j\alpha}r^{Q-1}\:dr\\
&=&\frac{\alpha^Q}{Q}\sum_{j=1}^{\infty}\phi(2^j\alpha)(2^{jQ}-2^{(j-1)Q})=\frac{(2^Q-1)\alpha^Q}{2^{2Q}Q}\sum_{j=1}^{\infty}\phi(2^j\alpha)2^{(j+1)Q}.
\end{eqnarray*}
Equation (\ref{polar}) and integrability of $\phi$ now imply that
\begin{equation}\label{series}
\sum_{j=1}^{\infty}\phi(2^j\alpha)2^{(j+1)Q}<\infty.
\end{equation}
Let us get back to the inequality (\ref{ntmax}). We get the following by making use of (\ref{mumonotone1}) and (\ref{mumonotone2}) in (\ref{ntmax}).
\begin{eqnarray}
\nonumber\mu\ast\phi_t(\xi)&\leq&\phi(0)t^{-Q}\mu(B(x_0,2C_{\mathcal{L}}\alpha t))+t^{-Q}\sum_{j=1}^{\infty}\phi(2^{j-1}\alpha)\mu(B(x_0,(2^j+1)C_{\mathcal{L}}\alpha t))\\\nonumber&=&\phi(0)(2C_{\mathcal{L}}\alpha)^Qm(B(0,1))\frac{\mu(B(x_0,2C_{\mathcal{L}}\alpha t)}{m(B(x_0,2C_{\mathcal{L}}\alpha t)}\\
\nonumber&&\:\:\:\:\:\:+\sum_{j=1}^{\infty}\phi(2^{j-1}\alpha)m(B(0,1))(C_{\mathcal{L}}\alpha)^Q(2^j+1)^Q\frac{\mu(B(x_0,(2^j+1)C_{\mathcal{L}}\alpha t))}{m(B(x_0,(2^j+1)C_{\mathcal{L}}\alpha t))}\\
\nonumber&\leq&m(B(0,1))(C_{\mathcal{L}}\alpha)^Q\left(2\phi(0)+\sum_{j=1}^{\infty}\phi(2^{j-1}\alpha)2^{(j+1)Q}\right)M_{HL}(\mu)(x_0).
\end{eqnarray}
In view of (\ref{series}), we see that the series inside the bracket above is finite and so we can write
\begin{equation*}
\mu\ast\phi_t(\xi)\leq c_{\alpha,\phi}M_{HL}(\mu)(x_0),\:\:\:\text{where}
\end{equation*}
 \begin{equation*}
 c_{\alpha,\phi}=m(B(0,1))(C_{\mathcal{L}}\alpha)^Q\left(2\phi(0)+\sum_{j=1}^{\infty}\phi(2^{j-1}\alpha)2^{(j+1)Q}\right)<\infty.
 \end{equation*}
  Taking supremum over all $(\xi,t)\in G\times(0,\infty)$ with $d_{\mathcal{L}}(x_0,\xi)<\alpha t$, we obtain
\begin{equation}\label{nontangentialmax}
M_{\phi}^{\alpha}\mu(x_0)\leq c_{\alpha,\phi}M_{HL}(\mu)(x_0).
\end{equation}
\end{proof}
\begin{thm}\label{heatmaximal}
Let $\mu\in M$ be a positive measure and let $x_0\in G$. Then for each $\alpha>0$, there exists positive constants $c_n$ and $c_{\alpha}$ (independent of $x_0$) such that 
\begin{equation}\label{heatmaxineq}
c_nM_{HL}(\mu)(x_0)\leq\sup_{t>0}\Gamma\mu(x_0,t^2)\leq\sup_{(x,t)\in \texttt{P}(x_0,\alpha)}\Gamma\mu(x,t)\leq c_{\alpha}M_{HL}(\mu)(x_0).
\end{equation}
\end{thm}
\begin{proof}
The second inequality is trivial as $\{(x_0,t^2):t>0\}\subset\texttt{P}(x_0,\beta)$ for every $\beta>0$. To prove the first inequality, we take 
\begin{equation*}
\phi(x)=c_0^{-1}\exp\left(-c_0d_{\mathcal{L}}(x)^2\right),\:\:\:x\in G.
\end{equation*}
Clearly, $\phi$ satisfies the hypothesis of Lemma \ref{maximal}. By the first part of the Gaussian estimate (\ref{gaussian}), we have 
\begin{equation*}
\mu\ast\phi_t(x)\leq\Gamma\mu(x,t^2),\:\:\:\text{for all}\:(x,t)\in G\times(0,\infty).
\end{equation*}
Applying first inequality of Lemma \ref{maximal}, we obtain 
\begin{equation*}
c_nM_{HL}(\mu)(x_0)\leq\sup_{t>0}\mu\ast\phi_t(x_0)\leq\sup_{t>0}\Gamma\mu(x_0,t^2),
\end{equation*} 
for some constant positive $c_n$, independent of $x_0$. On the other hand, we consider
\begin{equation*}
\psi(x)=c_0\exp\left(-\frac{d_{\mathcal{L}}(x)^2}{c_0}\right),\:\:\:x\in G.
\end{equation*}
It is obvious that $\psi$ satisfies the hypothesis of Lemma \ref{maximal}. Moreover, by the last part of the Gaussian estimate (\ref{gaussian}), we have 
\begin{equation}\label{max1}
\Gamma\mu(x,t)\leq\mu\ast\psi_{\sqrt{t}}(x),\:\:\:\text{for all}\:(x,t)\in G\times(0,\infty).
\end{equation}
But the last inequality of Lemma \ref{maximal} gives us
\begin{equation*}
\sup_{\substack{(\xi,t)\in G\times(0,\infty)\\d_{\mathcal{L}}(x_0,\xi)<\alpha \sqrt{t}}}\mu\ast\psi_{\sqrt{t}}(\xi)\leq c_{\alpha}M_{HL}(\mu)(x_0),
\end{equation*}
for some positive constant $c_{\alpha}$, independent of $x_0$. Using (\ref{max1}) and recalling the definition of $\texttt{P}(x_0,\alpha)$ (see Definition \ref{impdefnc}, i)), we note that 
\begin{equation*}
\sup_{(\xi,t)\in \texttt{P}(x_0,\alpha)}\Gamma\mu(\xi,t)=\sup_{\substack{(\xi,t)\in G\times(0,\infty)\\d_{\mathcal{L}}(x_0,\xi)<\alpha \sqrt{t}}}\Gamma\mu(\xi,t)\leq\sup_{\substack{(\xi,t)\in G\times(0,\infty)\\d_{\mathcal{L}}(x_0,\xi)<\alpha \sqrt{t}}}\mu\ast\psi_{\sqrt{t}}(\xi)
\end{equation*}
Hence,
\begin{equation*}
\sup_{(\xi,t)\in \texttt{P}(x_0,\alpha)}\Gamma\mu(\xi,t)\leq c_{\alpha}M_{HL}(\mu)(x_0).
\end{equation*} This completes the proof.
\end{proof}
To prove our main result we will also need an analogue of Montel's theorem for solutions of the heat equation (\ref{heateq}). We have already observed that the heat operator $\mathcal{H}=\frac{\partial}{\partial t}-\mathcal{L}$ is hypoelliptic on $G\times(0,\infty)$. Using this hypoellipticty, one can get a Montel-type result for solutions of the heat equation (\ref{heateq}) from a very general theorem proved in \cite[Theorem 4]{B}.
\begin{lem}\label{montelc}
Let $\{u_j\}$ be a sequence of solutions of the heat equation (\ref{heateq}) on $G\times(0,\infty)$. If $\{u_j\}$ is locally bounded then it has a subsequence which converges normally to a function $v$, defined on $G\times(0,\infty)$, which is also a solution of the heat equation (\ref{heateq}).
\end{lem}
We have already mentioned in the inroduction that the positive solutions of the classical heat equation on the Euclidean upper half space $\R^{n+1}_+$ are given by convolution of positive measures with the Euclidean heat kernel. In case of the heat equations on stratified Lie groups, we also have similar representation formula.
\begin{lem}\label{existence}(\cite[Lemma 2.3]{BU})
Let $u$ be a positive solution of the heat equation $\mathcal{H}u=0$ in the strip $G\times(0,T)$. Then, there exists a unique positive measure $\mu$ on $G$ such that
\begin{equation*}
u(x,t)=\int_G\Gamma(\xi^{-1}\circ x,t)\:d\mu(\xi),\:\:\:(x,t)\in G\times(0,T).
\end{equation*}
In this case we say that $\mu$ is the boundary measure of $u$.
\end{lem}
Bonfiglioli-Uguzzoni proved the above Lemma under the implicit assumption that $T\in(0,\infty)$. But the same proof will work for the case $T=\infty$. This type of theorem has been known as Widder-type representation fromula in the literature.

Given a function $F$ on $G\times(0,\infty)$ and $r>0$, we define the parabolic dilation of $F$ as
\begin{equation}\label{dilatedfc}
F_r(x,t)=F(\delta_r(x),r^2t),\:(x,t)\in G\times(0,\infty).
\end{equation}
\begin{rem}\label{dilatefc}
The notion of parabolic dilation is crucial for us primarily because of the following reasons.
\begin{enumerate}
\item[i)] If $F$ is a solution of the heat equation then so is $F_r$
for every $r>0$. Indeed, $\mathcal{L}$ is homogeneous of degree two with respect to the dilations $\{\delta_s\}_{s>0}$. This implies that
\begin{equation*}
\left(\mathcal{L}-\frac{\partial}{\partial t}\right)F(\delta_r(x),r^2t)=r^2\mathcal{L}F(\delta_r(x),r^2t)-r^2\frac{\partial}{\partial t}F(\delta_r(x),r^2t)=0.
\end{equation*}
\item[ii)] $(x,t)\in \texttt{P}(0,\alpha)$ if and only if $(\delta_r(x),r^2t)\in \texttt{P}(0,\alpha)$ for every $r>0$. 
\end{enumerate}
\end {rem}
Given $\nu\in M$ and $r>0$, we also define the dilate $\nu_r$ of $\nu$ by
\begin{equation}\label{dilatemc}
\nu_r(E)=r^{-Q}\nu\left(\delta_r(E)\right),
\end{equation}
for every Borel set $E\subset G$. We now prove a simple lemma involving the above dilates which will be used in the proof of our main result.
\begin{lem}\label{dilatec}
If $\nu\in M$, then for every $r>0$,  $\Gamma(\nu_r)=(\Gamma\nu)_r$, that is, for all $(x,t)\in G\times(0,\infty)$,
\begin{equation*}
\Gamma(\nu_r)(x,t)=\Gamma\nu\left(\delta_r(x),r^2t\right).
\end{equation*}
\end{lem}
\begin{proof}
For $E\subset G$ a Borel set, using (\ref{dilatemc}) it follows that
\begin{equation*}
\int_{G}\chi_E\:d\nu_r=r^{-Q}\nu\left(\delta_r(E)\right)=r^{-Q}\int_{G}\chi_{\delta_r(E)}(x)\:d\nu(x)=r^{-Q}\int_{G}\chi_{E}\left(\delta_{r^{-1}}(x)\right)\:d\nu(x).
\end{equation*}
Hence, for all nonnegative measurable functions $f$ we have  
\begin{equation*}
\int_{G}f(x)\:d\nu_r(x)=r^{-Q}\int_{G}f\left(\delta_{r^{-1}}(x)\right)\:d\nu(x).
\end{equation*}
It now follows from the above relation and from Theorem \ref{fundamental}, (iii) that for all $(x,t)\in G\times(0,\infty)$,
\begin{eqnarray*}
\Gamma(\nu_r)(x,t)&=&\int_G\Gamma(\xi^{-1}\circ x,t)\:d\nu_r(\xi)\\
&=&r^{-Q}\int_G\Gamma\left(\left(\delta_{r^{-1}}(\xi)\right)^{-1}\circ x,t\right)\:d\nu(x)\\
&=&r^{-Q}\int_G\Gamma\left(\delta_{r^{-1}}\left(\xi^{-1}\circ\delta_r(x)\right),r^{-2}r^2t\right)\:d\nu(x)\\&=&r^{-Q}r^Q\int_G\Gamma(\xi^{-1}\circ \delta_r(x),r^2t)\:d\nu(x)\\&=&(\Gamma\nu)_r(x,t).
\end{eqnarray*} 
\end{proof}

\section{Main theorem}
We shall first prove a special case of our main result. The proof of the main result will follow by reducing matters to this special case.
\begin{thm}\label{specialthc}
Suppose $u$ is a positive solution of the heat equation (\ref{heateq}), that is,
\begin{equation*}
\mathcal{H}u(x,t)=0,\:(x,t)\in G\times(0,\infty),
\end{equation*}
and $L\in[0,\infty)$. If a finite measure $\mu$ is the boundary measure of $u$ then the following statements are equivalent.
\begin{enumerate}
\item[(i)] $u$ has parabolic limit $L$ at $0$.
\item[(ii)]$\mu$ has strong derivative $L$ at $0$.
\end{enumerate} 
\end{thm}
\begin{proof}
We first prove that (i) implies (ii). We choose a $d_{\mathcal{L}}$-ball $B_0\subset G$, a sequence of positive numbers $\{r_j\}$ converging to zero and consider the quotient
\begin{equation*}
L_j=\frac{\mu\left(\delta_{r_j}(B_0)\right)}{m\left(\delta_{r_j}(B_0)\right)}.
\end{equation*}
Assuming (i), we will prove that $\{L_j\}$ is a bounded sequence and every convergent subsequence of $\{L_j\}$ converges to $L$. We first choose a positive real number $s$ such that $B_0$ is contained in the $d_{\mathcal{L}}$-ball $B(0,s)$. Then for all $j\in\N$,
\begin{equation}\label{lj}
L_j\leq \frac{\mu\left(\delta_{r_j}(B(0,s))\right)}{m\left(\delta_{r_j}(B_0)\right)}=\frac{\mu\left(\delta_{r_j}(B(0,s))\right)}{m\left(\delta_{r_j}(B(0,s))\right)}\times\frac{m(B(0,s))}{m(B_0)}\leq \frac{m(B(0,s))}{m(B_0)}M_{HL}(\mu)(0).
\end{equation}
Since $\mu$ is the boundary measure for $u$ we have that
\begin{equation*}
u(x,t)=\Gamma\mu (x,t),\:\:\:\:\:\text{for all $(x,t)\in G\times(0,\infty)$.}
\end{equation*}	
By hypothesis, $u(0,t^2)$ converges to $L$ as $t$ tends to zero which implies, in particular, that there exists a positive number $\beta$ such that
\begin{equation*}
\sup_{t<\beta}\Gamma\mu(0,t^2)<\infty.
\end{equation*}
Since $\mu$ is a finite measure, using (\ref{gaussian}) we also have
\begin{equation*}
\Gamma\mu(0,t^2)\leq c_0t^{-Q}\int_{G}\exp\left(-\frac{d_{\mathcal{L}}(x)^2}{c_0t^2}\right)\:d\mu(x)\leq c_0t^{-Q}\mu(G)\leq c_0\beta^{-Q}\mu(G),
\end{equation*}
for all $t\geq\beta$ and hence
\begin{equation*}
\sup_{0< t<\infty}\Gamma\mu(0,t^2)<\infty.
\end{equation*}
Inequality (\ref{heatmaxineq}), now implies that $M_{HL}(\mu)(0)$ is finite. Boundedness of the sequence $\{L_j\}$ is now a consequence of the inequality (\ref{lj}). We choose a convergent subsequence of $\{L_j\}$ and denote it also, for the sake of simplicity, by $\{L_j\}$. For $j\in\N$ we define
\begin{equation*}
u_j(x,t)=u\left(\delta_{r_j}(x),r_j^2t\right),\:\:\:\:(x,t)\in G\times(0,\infty).
\end{equation*}
Then by Remark \ref{dilatefc}, i), $\{u_j\}$ is a sequence of solutions of the heat equation in $G\times(0,\infty)$. We claim that $\{u_j\}$ is locally bounded. To see this, we choose a compact set  $K\subset G\times(0,\infty)$. Then there exists a positive number $\alpha$ such that $K$ is contained in the parabolic region $\texttt{P}(0,\alpha)$. Indeed, we consider the map
\begin{equation*}
(x,t)\mapsto\frac{d_{\mathcal{L}}(x)}{\sqrt{t}},\:\:\:(x,t)\in G\times(0,\infty),
\end{equation*}
that is, $K\subset\texttt{P}(0,\alpha)$. Since $d_{\mathcal{L}}$ is continuous on $G$, this map is also continuous. As $K$ is compact, image of $K$ under this map is bounded and hence there exists a positive real number $\alpha$ such that 
\begin{equation*}
\frac{d_{\mathcal{L}}(x)}{\sqrt{t}}<\alpha,\:\:\:\:\text{for all $(x,t)\in K$.}
\end{equation*}
Using the invariance of $\texttt{P}(0,\alpha)$ under the parabolic dilation (see Remark \ref{dilatefc}, ii)) and (\ref{heatmaxineq}), it follows that for all $j\in\N$
\begin{equation*}
\sup_{(x,t)\in \texttt{P}(0,\alpha)}u_j(x,t)\leq \sup_{(x,t)\in \texttt{P}(0,\alpha)}u(x,t)\leq c_{\alpha}M_{HL}(\mu)(0).
\end{equation*}
Hence, $\{u_j\}$ is locally bounded. Lemma \ref{montelc}, now guarantees the existence of a subsequence $\{u_{j_k}\}$ of $\{u_j\}$ which converges normally to a positive solution $v$ of the heat equation in $G\times(0,\infty)$. We claim that for all $(x,t)\in G\times(0,\infty)$
\begin{equation}\label{vlimitmain}
v(x,t)=L.
\end{equation}
To see this, we take  $(x_0,t_0)\in G\times(0,\infty)$ and choose $\eta>0$ such that $(x_0,t_0)\in\texttt{P}(0,\eta)$. Since $\{r_{j_k}\}$ converges to zero as $k$ goes to infinity and $u(x,t)$ has limit $L$, as $(x,t)$ tends to $(0,0)$ within $\texttt{P}(0,\eta)$,
\begin{equation*}
v(x_0,t_0)=\lim_{k\to\infty}u_{j_k}(x_0,t_0)=\lim_{k\to\infty}u\left(\delta_{r_{j_k}}(x_0),r_{j_k}^2t_0\right)=L,
\end{equation*}
as $(\delta_{r_{j_k}}(x),r_{j_k}^2t)\in \texttt{P}(0,\eta)$ for all $j_k\in\N$. This settles the claim. On the other hand, by Lemma \ref{dilatec}, we have for all $(x,t)\in G\times(0,\infty)$
\begin{equation}\label{udilatec}
u_{j_k}(x,t)=u\left(\delta_{r_{j_k}}(x),r_{j_k}^2t\right)=\Gamma\mu\left(\delta_{r_{j_k}}(x),r_{j_k}^2t\right)=\Gamma\left(\mu_{r_{j_k}}\right)(x,t).
\end{equation}
It follows from (\ref{vlimitmain}) and (\ref{udilatec}) that $\{\Gamma(\mu_{r_{j_k}})\}$ converges normally to $L\equiv\Gamma(Lm)$. Then, Lemma \ref{normalc} implies that the sequence of measures $\{\mu_{r_{j_k}}\}$ converges to $Lm$ in weak* and hence by Lemma \ref{mthc}, $\{\mu_{r_{j_k}}(B)\}$ converges to $Lm(B)$ for every  $d_{\mathcal{L}}$-ball $B\subset G$.
Using this for $B=B_0$, we get that
\begin{equation*}
Lm(B_0)=\lim_{k\to \infty}\mu_{r_{j_k}}(B_0)=\lim_{k\to\infty}{r_{j_k}}^{-Q}\mu\left(\delta_{{r_{j_k}}}(B_0)\right)=m(B_0)\lim_{k\to\infty}\frac{\mu\left(\delta_{{r_{j_k}}}(B_0)\right)}{m\left(\delta_{{r_{j_k}}}(B_0)\right)}.
\end{equation*}
This implies that the sequence $\{L_{j_{k}}\}$ converges to $L$ and hence so does $\{L_j\}$, as $\{L_j\}$ is convergent. Thus, every convergent subsequence of  the bounded sequence $\{L_j\}$ converges to $L$. This implies that $\{L_j\}$ itself converges to $L$. Since $B_0$ and $\{r_j\}$ is arbitrary, $\mu$ has strong derivative $L$ at $0$.

Now, we prove (ii) implies (i). We suppose that the strong derivative of $\mu$ at zero is equal to $L$ but the parabolic limit of $u$ at zero is not equal to $L$. Then there exists a positive number $\alpha$ and a sequence $\{(x_j,t_j^2)\}\subset \texttt{P}(0,\alpha)$ with $(x_j,t_j^2)$ converging to $(0,0)$ but $\{u(x_j,t_j^2)\}$ fails to converge to $L$. Since $D\mu(0)$ exists finitely (in fact, equal to $L$) it follows, in particular, that
\begin{equation*}
\sup_{0<r<\delta}\frac{\mu(B(0,r))}{m(B(0,r))}<L+1,
\end{equation*}
for some $\delta>0$. Finiteness of the measure $\mu$ implies that
\begin{equation*}
\frac{\mu(B(0,r))}{m(B(0,r))}\leq\frac{\mu(G)}{m(B(0,1))r^Q}\leq\frac{\mu(G)}{m(B(0,1))\delta^Q},
\end{equation*}
for all $r\geq\delta$. The above two inequalities together with (\ref{heatmaxineq}) shows that
\begin{equation*}
\sup_{(x,t)\in \texttt{P}(0,\alpha)}u(x,t)\leq c_{\alpha}M_{HL}(\mu)(0)<\infty.
\end{equation*}
In particular, $\{u(x_j,t_j^2)\}$ is a bounded sequence. We now consider a convergent subsequence of this sequence, denote it also, for the sake of simplicity, by $\{u(x_j,t_j^2)\}$ such that
\begin{equation}\label{limitLc}
\lim_{j\to\infty} u(x_j,t_j^2)=L'.
\end{equation}
We will prove that $L'$ is equal to $L$. Using the sequence $\{t_j\}$ we consider the dilates
\begin{equation*}
u_j(x,t)=u\left(\delta_{t_j}(x),t_j^2t\right),\:\:\:\:(x,t)\in G\times(0,\infty).
\end{equation*}
Arguments used in the first part of the proof shows that $\{u_j\}$ is a locally bounded sequence of nonnegative solutions of the heat equation in $G\times(0,\infty)$. Hence, by Lemma \ref{montelc}, there exists a subsequence $\{u_{j_k}\}$ of $\{u_j\}$ which converges normally to a positive solution $v$ of the heat equation in $G\times(0,\infty)$. Therefore, Lemma \ref{existence} shows that there exists $\nu\in M$ such that $v$ equals $\Gamma\nu$. We now consider the sequence of dilates $\{\mu_k\}$ of $\mu$ by $\{t_{j_k}\}$ according to (\ref{dilatemc}). An application of Lemma \ref{dilatec} then implies that $\Gamma\mu_k=u_{j_k}$. It follows that the sequence of functions $\{\Gamma\mu_k\}$ converges normally to $\Gamma\nu$. By Lemma \ref{normalc}, we thus obtain weak* convergence of $\{\mu_k\}$ to $\nu$.

Since $D\mu(0)=L$, it follows that for any $d_{\mathcal{L}}$-ball $B\subset G$,
\begin{equation*}
\lim_{k\to\infty}\mu_k(B)=\lim_{k\to\infty}{t_{j_k}}^{-Q}\mu(\delta_{{t_{j_k}}}(B))=\lim_{k\to\infty}\frac{\mu(\delta_{{t_{j_k}}}(B))}{m(\delta_{{t_{j_k}}}(B)}m(B)=Lm(B).
\end{equation*}
Hence by Lemma \ref{mthc}, $\nu=Lm$. As $v=\Gamma\nu$, it follows that
\begin{equation*}
v(x,t)=L,\:\:\:\:\text{for all $(x,t)\in G\times(0,\infty)$}.
\end{equation*}
This, in turn, implies that $\{u_{j_k}\}$ converges to the constant function $L$ normally in $G\times(0,\infty)$. On the other hand, we note that
\begin{equation*}
u(x_{j_k},t_{j_k}^2)=u\left(\delta_{t_{j_k}}\left(\delta_{{t^{-1}_{j_k}}}(x_{j_k})\right),t_{j_k}^2\right)=u_{j_{k}}\left(\delta_{{t^{-1}_{j_k}}}(x_{j_k}),1\right).
\end{equation*}
Since $(x_{j_k},t_{j_k}^2)$ belongs to the parabolic region $\texttt{P}(0,\alpha)$, for all $k\in\N$, it follows that
\begin{equation*}
\left(\delta_{{t^{-1}_{j_k}}}(x_{j_k}),1\right)\in\overline B(0,\alpha)\times\{1\},
\end{equation*}
which is a compact subset of $G\times(0,\infty)$. Therefore,
\begin{equation*}
\lim_{k\to\infty}u(x_{j_k},t_{j_k}^2)=L.
\end{equation*}
In view of (\ref{limitLc}) we can thus conclude that $L'$ equals $L$. So, every convergent subsequence of the original sequence $\{u(x_j,t_j^2)\}$ converges to $L$. This contradicts our assumption that $\{u(x_j,t_j^2)\}$ fails to converge to $L$. This completes the proof.
\end{proof}
Now, we are in a position to state and prove our main result.
\begin{thm}\label{mainc}
Suppose $u$ is a positive solution of the heat equation $\mathcal{H}u=0$ in $G\times(0,T)$, for some $0<T\leq\infty$ and suppose $x_0\in G$, $L\in[0,\infty)$. If $\mu$ is the boundary measure of $u$ then the following statements are equivalent.
\begin{enumerate}
\item[(i)] $u$ has parabolic limit $L$ at $x_0$.
\item[(ii)]$\mu$ has strong derivative $L$ at $x_0$.
\end{enumerate} 
\end{thm}
\begin{proof}
We consider the translated measure $\mu_0=\tau_{x_0}\mu$, where
\begin{equation*}
\tau_{x_0}\mu (E)=\mu(x_0\circ E),
\end{equation*}
for all Borel subsets $E\subset G$. Using translation invariance of the Lebesgue measure $m$, it follows from the definition of strong derivative that $D\mu_0(0)$ and $D\mu
(x_0)$ are equal. Since $\Gamma\mu_0$ is given by the convolution of $\mu_0$ with $\gamma_{\sqrt{t}}$ and translation commutes with convolution, it follows that
\begin{equation}\label{trans}
\Gamma\mu_0(x,t)= (\gamma_{\sqrt{t}}\ast \tau_{x_0}\mu )(x)=\tau_{x_0}(\gamma_{\sqrt{t}}\ast\mu )(x)=\Gamma\mu(x_0\circ x,t).
\end{equation}
We fix an arbitrary positive number $\alpha$. As $(x,t)\in \texttt{P}(0,\alpha)$ if and only if $(x_0\circ x,t)\in \texttt{P}(x_0,\alpha)$, one infers from  (\ref{trans}) that
\begin{equation*}
\lim_{\substack{(x,t)\to(0,0)\\(x,t)\in \texttt{P}(0,\alpha)}}\Gamma\mu_0(x,t)=\lim_{\substack{(\xi,t)\to(x_0,0)\\(\xi,t)\in \texttt{P}(x_0,\alpha)}}\Gamma\mu(\xi,t).
\end{equation*}
Hence, it suffices to prove the theorem under the assumption that $x_0=0$. We now show that we can even take $\mu$ to be a finite measure. Let $\tilde{\mu}$ be the restriction of $\mu$ on the $d_{\mathcal{L}}$-ball $B(0,C_{\mathcal{L}}^{-1})$. Suppose $B(y,s)$ is any given $d_{\mathcal{L}}$-ball. Then for all $0<r<\left(C_{\mathcal{L}}^2(s+d_{\mathcal{L}}(y))\right)^{-1}$, it follows that whenever $\xi\in\delta_r(B(y,s))=B(\delta_r(y),rs)$, we have
\begin{equation*}
d_{\mathcal{L}}(0,\xi)\leq C_{\mathcal{L}}\left(d_{\mathcal{L}}(0,\delta_r(y))+d_{\mathcal{L}}(\delta_r(y),\xi)\right)\leq C_{\mathcal{L}}\left(rd_{\mathcal{L}}(y)+rs\right)<C_{\mathcal{L}}^{-1}.
\end{equation*}
In other words, $\delta_r(B(y,s))$ is a subset of $B(0,C_{\mathcal{L}}^{-1})$. This in turn implies that $D\mu(0)$ and $D\tilde{\mu}(0)$ are equal. We now claim that
\begin{equation}\label{finaleqc}
\lim_{\substack{(x,t)\to(0,0)\\(x,t)\in \texttt{P}(0,\alpha)}}\Gamma\mu(x,t)=\lim_{\substack{(x,t)\to(0,0)\\(x,t)\in \texttt{P}(0,\alpha)}}\Gamma\tilde{\mu}(x,t).
\end{equation}
In this regard, we first observe that
\begin{equation*}
\lim_{t\to 0}\int_{G\setminus B(0,C_{\mathcal{L}}^{-1})}\Gamma(\xi^{-1}\circ x,t)\:d\mu(\xi)=0,
\end{equation*}uniformly for $x\in B(0,1/(2C_{\mathcal{L}}^2))$. Indeed, for $x\in B(0,1/(2C_{\mathcal{L}}^2))$ and $\xi\in G\setminus B(0,C_{\mathcal{L}}^{-1})$, we have 
\begin{equation*}
d_{\mathcal{L}}(\xi^{-1}\circ x)\geq\frac{1}{C_{\mathcal{L}}}d_{\mathcal{L}}(\xi)-d_{\mathcal{L}}(x)\geq\frac{d_{\mathcal{L}}(\xi)}{C_{\mathcal{L}}}-\frac{d_{\mathcal{L}}(\xi)}{2C_{\mathcal{L}}}=\frac{d_{\mathcal{L}}(\xi)}{2C_{\mathcal{L}}}\geq\frac{1}{2C_{\mathcal{L}}^2}.
\end{equation*}
Using the Gaussian estimate (\ref{gaussian}) and the inequality above, we get
\begin{eqnarray*}
&&\int_{G\setminus B(0,C_{\mathcal{L}}^{-1})}\Gamma(\xi^{-1}\circ x,t)\:d\mu(\xi)\\
&\leq& c_0t^{-\frac{Q}{2}}\int_{G\setminus B(0,C_{\mathcal{L}}^{-1})}\exp\left(-\frac{d_{\mathcal{L}}(\xi^{-1}\circ x)^2}{c_0t}\right)\:d\mu(\xi)\\
&\leq&c_0t^{-\frac{Q}{2}}\int_{G\setminus B(0,C_{\mathcal{L}}^{-1})}\exp\left(-\frac{d_{\mathcal{L}}(\xi)^2}{4c_0C_{\mathcal{L}}^2t}\right)\:d\mu(\xi)\\
&\leq& c_0t^{-\frac{Q}{2}}\exp\left(-\frac{1}{8c_0C_{\mathcal{L}}^4t}\right)\int_{G\setminus B(0,C_{\mathcal{L}}^{-1})}\exp\left(-\frac{d_{\mathcal{L}}(\xi)^2}{8c_0C_{\mathcal{L}}^2t}\right)\:d\mu(\xi)\\
&\leq& c_0t^{-\frac{Q}{2}}\exp\left(-\frac{1}{8c_0C_{\mathcal{L}}^4t}\right)\int_{G\setminus B(0,C_{\mathcal{L}}^{-1})}\exp\left(-\frac{2c_0d_{\mathcal{L}}(\xi)^2}{t_0}\right)\:d\mu(\xi),
\end{eqnarray*}
for all $t<(16c_0^2C_{\mathcal{L}}^2)^{-1}t_0$, where $t_0$ is a fixed positive number less than $T$.  As $\Gamma\mu(0,t_0/2)$ exists, the Gaussian estimate (\ref{gaussian}) implies that the integral on the right-hand side in the last inequality is finite. Hence, letting $t$ goes to zero on the right-hand side in the last inequality, the observation follows. Now, 
\begin{eqnarray*}
\Gamma\mu(x,t)&=&\int_{B(0,C_{\mathcal{L}}^{-1})}\Gamma(\xi^{-1}\circ x,t)\:d\mu(\xi)+\int_{G\setminus B(0,C_{\mathcal{L}}^{-1})}\Gamma(\xi^{-1}\circ x,t)\:d\mu(\xi)\\
&=& \Gamma\tilde{\mu}(x,t)+ \int_{G\setminus B(0,C_{\mathcal{L}}^{-1})}\Gamma(\xi^{-1}\circ x,t)\:d\mu(\xi).
\end{eqnarray*} 
Given any $\epsilon>0$, we get some $t_1\in(0,(16c_0^2C_{\mathcal{L}}^2)^{-1}t_0)$ such that for all $t\in (0,t_1)$, the integral on the right-hand side of the equality above is smaller than $\epsilon$ for all $x\in B(0,1/(2C_{\mathcal{L}}^2))$. On the other hand, if we choose $t\in (0,1/(4\alpha^2C_{\mathcal{L}}^4))$ then it follows that
\begin{equation*}
\texttt{P}(0,\alpha)\cap \{(x,t)\mid t\in (0,1/(4\alpha^2C_{\mathcal{L}}^4)) \}\subset B(0,1/(2C_{\mathcal{L}}^2))\times (0,1/(4\alpha^2C_{\mathcal{L}}^4)).
\end{equation*}
Hence, for all $(x,t)\in \texttt{P}(0,\alpha)$ with $t\in(0,\min\{t_1,1/(4\alpha^2C_{\mathcal{L}}^4)\})$  we have
\begin{equation*}
|\Gamma\mu (x,t)-\Gamma\tilde{\mu}(x,t)|<\epsilon.
\end{equation*}
This proves (\ref{finaleqc}). Therefore, as $\alpha>0$ is arbitrary, we may and do suppose that $\mu$ is a finite measure. Using this, without loss of generality, we may also assume $T=\infty$. The proof now follows from Theorem \ref{specialthc}.
\end{proof}
\section*{Acknowledgements}
The author would like to thank Swagato K. Ray for many
useful discussions and suggestions during the course of this work. The author is supported by a research fellowship from Indian Statistical Institute.

\end{document}